\documentclass[12pt]{article}
\usepackage{amsmath,amssymb,amsfonts,amsthm,mathrsfs}
\usepackage[page,title,titletoc,header]{appendix}
\usepackage{enumerate}
\usepackage{color}
\usepackage{geometry}
\usepackage{hyperref}
\usepackage{algpseudocodex}
\usepackage{algorithm}
\usepackage[numbers,square]{natbib}
\usepackage{graphicx}
\usepackage{subcaption}
\usepackage{changes}
\definechangesauthor[name={Yu Chenhao}, color=blue]{Yu}

\hypersetup{hypertex=true,colorlinks=true,linkcolor=blue,citecolor=red}

\usepackage{makecell}
\usepackage{multirow}
\usepackage{array}
\usepackage{threeparttable}

\numberwithin{equation}{section}
\newtheorem{theorem}{Theorem}
\newtheorem{lemma}{Lemma}[section]
\newtheorem{proposition}{Proposition}[section]
\newtheorem{corollary}{Corollary}
\newtheorem{definition}{Definition}

\newtheorem{assumption}{Assumption}

\newtheoremstyle{case}{}{}{}{}{\bf}{:}{ }{}
\theoremstyle{case}

\newtheoremstyle{case1}{}{}{}{}{\bf}{:}{ }{}
\theoremstyle{case1}

\allowdisplaybreaks

\newcommand\keywords[1]{\textbf{Keywords: }#1}

\newcommand{\be}{\begin{equation}}
\newcommand{\ee}{\end{equation}}
\newcommand{\bea}{\begin{eqnarray*}}
\newcommand{\eea}{\end{eqnarray*}}
\usepackage{amsmath,amsfonts,bm}

\def\Algref#1{Algorithm~\ref{#1}}

\newcommand{\lemmaref}[1]{Lemma~\ref{#1}}
\newcommand{\corollaryref}[1]{Corollary~\ref{#1}}

\def\1{\bm{1}}

\def\rvz{{\mathbf{z}}}

\newcommand{\rvxi}[1]{\bm{\xi}_{#1}}
\newcommand{\rvup}[1]{\bm{\upsilon}_#1}

\def\va{{\bm{a}}}
\def\vb{{\bm{b}}}

\def\vg{{\bm{g}}}

\def\vw{{\bm{w}}}
\def\vx{{\bm{x}}}
\def\vy{{\bm{y}}}
\def\vz{{\bm{z}}}

\DeclareMathAlphabet{\mathsfit}{\encodingdefault}{\sfdefault}{m}{sl}
\SetMathAlphabet{\mathsfit}{bold}{\encodingdefault}{\sfdefault}{bx}{n}

\newcommand{\R}{\mathbb{R}}

\newcommand{\la}{\left\langle}
\newcommand{\ra}{\right\rangle}
\newcommand{\sumlimits}[2]{\sum\limits_{#1}^{#2}}
\newcommand{\fx}[1]{f\left(\vx_{#1}\right)}
\newcommand{\fxmd}[1]{f\left(\vx_{#1}^{md}\right)}
\newcommand{\fxag}[1]{f\left(\vx_{#1}^{ag}\right)}
\newcommand{\xag}[1]{\vx_{#1}^{ag}}
\newcommand{\xmd}[1]{\vx_{#1}^{md}}
\newcommand{\nablaf}[2]{\nabla f(\vx_{#1}^{#2})}
\newcommand{\norm}[1]{\left\|{#1}\right\|}
\newcommand{\triangleag}[1]{\triangle_{#1}^{ag}}
\newcommand{\trianglemd}[1]{\triangle_{#1}^{md}}
\newcommand{\withprobability}{with probability at least $1-\delta$}

\newcommand{\Fone}{\mathcal{F}_1}
\newcommand{\Ftwo}{\mathcal{F}_2}
\newcommand{\Fthree}{\mathcal{F}_3}

\newcommand{\PF}{\mathcal{P}(\Ftwo)}
\newcommand{\functionvaluegap}[2]{f\left(\vx_{#1}^{#2}\right)-f^*}
\newcommand{\LFone}{\mathcal{L}_1}

\newcommand{\logte}{\log\frac{T\mathrm{e}}{\delta}}

\newcommand{\twolayerL}{L_2}
\newcommand{\Ffour}{\mathcal{F}_4}

\newcommand{\Llow}{L_0+L_1\left(\norm{\nabla \fx{t}}+\norm{\nabla \fx{t}}^p\right)+L_2\left(\norm{\nabla \fx{t}}+\triangle_t^q\right)}

\newcommand{\Ldeterministic}{2\left(L_0+L_1\left(\left(g(\Fone)\right)^{\frac{1}{2}}+\left(g(\Fone)\right)^\frac{p}{2}\right)+L_2\left(\left(g(\Fone)\right)^{\frac{1}{2}}+\Fone^q\right)+8\mathcal{C}_1^2\left(L_1+L_2\right)^4 \right)}

\newcommand{\Ldescentlemma}{L_0+L_1 \left(g(\Ftwo)\right)^\frac{p}{2}+L_2 \Ftwo^q}

\newcommand{\nabPsi}[2]{\nabla \Psi(x_{#1}^{#2})}

\newcommand{\gapdistance}[2]{ A_{#1}\triangleag{#1}+\frac{2}{\beta}\|\vx_{#2}-\vx^*\|^2}

\begin{document}
\title{Convergence Analysis of Nesterov's Accelerated Gradient Descent under Relaxed Assumptions}

\author{Chenhao Yu\thanks{Chenhao Yu is with the Center for Data Science and School of Mathematical Sciences, Zhejiang University.},\and 
Junhong Lin\thanks{Junhong Lin is with the Center for Data Science, Zhejiang University.}
}
\date{\small First on openreview.net in Dec. 2025}

\maketitle

\begin{abstract}
We study convergence rates of Nesterov's Accelerated Gradient Descent (NAG) method for convex optimization in both deterministic and stochastic settings. We focus on a more general smoothness condition raised from several machine learning problems empirically and theoretically. We show the accelerated convergence rate of order $\mathcal{O}\left(1/T^2\right)$ in terms of the function value gap, given access to exact gradients of objective functions, matching the optimal rate for standard smooth convex optimization in \citep{nesterov1983method}. Under the relaxed affine-variance noise assumption for stochastic optimization, we establish the high-probability convergence rate of order $\tilde{\mathcal{O}}\left(\sqrt{\log\left(1/\delta\right)/T}\right)$ and this rate could improve to $\tilde{\mathcal{O}}\left(\log\left(1/\delta\right)/T^2\right)$ when the noise parameters are sufficiently small. Here, $T$ denotes the total number of iterations and $\delta$ is the probability margin. Up to logarithm factors, our probabilistic convergence rate reaches the same order of the expected rate obtained in \citep{ghadimi2016accelerated} where the assumptions of bounded variance noise and Lipschitz smoothness are required.
We additionally establish the convergence rates of the studied algorithms in the non-convex setting.
\end{abstract}


\section{Introduction}
In this paper, we consider the following classical unconstrained optimization problem,
\begin{align}
  \label{problem}
  \min\limits_{\vx\in\mathbb{R}^d} f(\vx),
\end{align}
where the objective function $f(\vx)$ is convex and can be potentially stochastic, i.e.,
\begin{align*}
  f(\vx)=\mathbb{E}_{\rvz \sim \mathcal{D}}[f_{\rvz}\left(\vx;\rvz\right)].
\end{align*} 
Here $\mathcal{D}$ is a probability distribution from which the random vector $\rvz$ is drawn.

%In this paper, we study the convergence rate for the well-known algorithm, Nesterov's Accelerated Gradient Descent (NAG) and its stochastic variants with convex objective function.

%In the convex optimization, we usually analyze the convergence rate in terms of the function value gap. 

Gradient-based algorithms \cite{robbins1951stochastic,nesterov1983method,nesterov2013introductor,duchi2011adaptive} play an important role in solving (\ref{problem}).
%As usual, one typically focuses on the function value gap for convex objectives and the squared gradient norm for non-convex ones. 
As usual, one typically focuses on the function value gap for convex objectives and the squared gradient norm for general non-convex ones.\footnote{ 
An extensive literature on minimizing structured non-convex functions focuses on the function value gap. Examples include work on Polyak-Łojasiewicz functions \cite{karimi2016linear}, (strongly) quasar-convex functions \cite{hinder2020near} and (strongly) quasiconvex functions \cite{grad2025strongly}. This is beyond the discussion of this paper.}
In the deterministic setting with access to the exact gradient $\nabla f(\vx)$, Gradient Descent (GD) achieves a convergence rate of $\mathcal{O}(1/T)$ for smooth convex functions \cite{nesterov2013introductor}, whereas for smooth non-convex functions, the rate of the same order is obtained for the squared gradient norm.
Here, $T$ is the total number of iterations.
The convergence rate for smooth convex optimization can be improved to $\mathcal{O}(1/T^2)$ using Nesterov's Accelerated Gradient Descent (NAG), as established in the seminal work \cite{nesterov1983method}. Furthermore, this complexity bound is known to be optimal among gradient based algorithms \cite{nemirovskij1983problem}, without further assumptions.
%In the convex setting, the rate can be improved to $\mathcal{O}(1/T^2)$ by studying Nesterov's Accelerated Gradient Descent (NAG) in a seminal work \cite{nesterov1983method} and this complexity bound is not improvable for smooth convex optimization when $d$ is large enough \cite{nemirovskij1983problem}.

For stochastic optimization where only the gradient estimator is accessible, Stochastic Gradient Descent (SGD) \cite{robbins1951stochastic}  is commonly used.
%and its variants \cite{ward2020adagrad,wang2023convergence2} are commonly used.
\cite{lan2012optimal} provided an expected upper bound of order $\mathcal{O}\left(1/T+\sigma/\sqrt{T}\right)$ for convex objective functions and \cite{ghadimi2013stochastic} obtained the bound of the same order for the non-convex case, both of them assuming bounded variance noise with noise parameter $\sigma$ and smooth objective functions. 
This bound is optimal in the non-convex setting since it matches the lower bound in \cite{arjevani2023lower}.
To study the acceleration behavior in the stochastic convex optimization, \cite{lan2012optimal,ghadimi2016accelerated} explored (and generalized) stochastic NAG (SNAG) and obtained the expected convergence rate of order $\mathcal{O}\left(1/T^2+\sigma/\sqrt{T}\right)$ for smooth objective functions, which in general cannot be improved in the same setting \cite{nemirovskij1983problem,lan2012optimal}.

Although much theoretical progress has been made on gradient-based algorithms, most of these analysis required Lipschitz smoothness condition \cite{ghadimi2013stochastic,ghadimi2016accelerated,levy2018online,ward2020adagrad,attia2023sgd}, i.e., $\exists L>0$, such that 
\begin{align*}
\norm{\nabla f\left(\vx\right)-\nabla f\left(\vy\right)}\leq L \norm{\vx-\vy}, \forall \vx, \vy \in\mathbb{R}^d,
\end{align*}
or equivalently $\norm{\nabla^2 f\left(\vx\right)}\leq L, \forall \vx \in \mathbb{R}^d$ for twice-differentiable functions.
Recently, several researchers have found evidence that this condition is not satisfied by many important machine learning models \cite{chen2023generalized}, such as neural network models \cite{zhang2019gradient} and distributionally robust optimization \cite{jin2021non}.
Based on empirical observations, \cite{zhang2019gradient} proposed $\left(L_0,L_1\right)$-smoothness condition, allowing $\norm{\nabla^2 f\left(\vx\right)}$ to grow linearly with respect to $\norm{\nabla f\left(\vx\right)}$, and later \cite{zhang2020improved} further relaxed this condition, not requiring the second differentiability of the objective function, i.e., there exist $L_0,L_1\geq0$, for any $\vx,\vy\in\mathbb{R}^d$, such that $\norm{\vx-\vy}\leq 1/L_1$,
\begin{align}
  \label{eq:gnlsm}
  \norm{\nabla f\left(\vx\right)- \nabla f\left(\vy\right)}\leq \left(L_0+L_1\norm{\nabla f\left(\vx\right)}\right) \norm{\vx-\vy}.
\end{align}

Based on this generalized smoothness condition, \cite{yu2025convergence} studied Randomized Stochastic Accelerated Gradient Descent (RSAG) proposed in \cite{ghadimi2016accelerated} and provided high-probability convergence rate of order $\tilde{\mathcal{O}}\left(1/T+\sigma/\sqrt{T}\right)$ for both convex and non-convex optimization (under sub-Gaussian relaxed affine-variance noise), which implies a gap between optimal rate obtained in the smooth convex optimization.
 Under a similar generalized smoothness condition, \cite{li2024convex} showed accelerated convergence rate of order $\mathcal{O}\left(1/T^2\right)$ for deterministic NAG  in convex optimization, and they also provided expected convergence rate of order $\mathcal{O}\left(\left(1+\sigma\right)/\sqrt{T}\right)$ for SGD in the non-convex stochastic optimization.
 To the best of our knowledge, it remains an open question whether SNAG can achieve an accelerated convergence rate of order $\tilde{\mathcal{O}}\left(1/T^2+ \sigma/\sqrt{T}\right)$ under the generalized smoothness condition for convex optimization. We believe that a proof for the stochastic setting presents certain challenges; in particular, the analysis for deterministic NAG by \cite{li2024convex} does not appear to be trivially extendable.

In this paper, we aim to close this gap,
developing the
accelerated convergence rate for SNAG under more generalized smoothness and relaxed affine-variance noises for stochastic convex optimization.
%
%
%However, \cite{li2024convex} does not include the research for SNAG, which is one of the motivations for our paper.
%As for other gradient-based algorithms, \cite{zhang2019gradient,zhang2020improved,gorbunov2024methods} studied gradient clipping method while \cite{faw2023beyond,wang2023convergence,hongrevisiting,yu2025convergence}  investigated adaptive gradient algorithms.
Specifically,  inspired by the theoretical examples in \cite{taheri2023fast} and \cite{chen2023generalized}, we focus on the following more general and practical smoothness condition.                                                         
\begin{definition}[$\left(L_0,L_1,L_2\right)$-smoothness\footnote{For the sake of rigor, we define $1/0=+\infty$ throughout the paper.}]
  \label{definition 1}
  Let $L_i \geq 0, \forall  1\leq  i\leq 3.$ 
    $f(\cdot)$ is $\left(L_0,L_1,L_2\right)$-smooth if and only if for any $\vx,\vy\in\R^d$ such that $\norm{\vx-\vy}\leq \min \left\{\frac{1}{L_1},\frac{1}{L_2}\right\}$,
    \begin{align}
        \norm{\nabla f(\vx)- \nabla f(\vy)}\leq \left(L_0+L_1\norm{\nabla f(\vx)}^p+L_2\left(f(\vx)-f^*\right)^q\right)\norm{\vx-\vy},
    \end{align}
    where $p\in[0,2)$ and $q\geq 0$.
\end{definition}
Obviously, Definition \ref{definition 1} covers a broader range of relaxed smoothness.
 %than the generalized smoothness proposed in \cite{zhang2019gradient,zhang2020improved} and, of course, the classic Lipschitz smoothness. 
 Particularly, it is situated between two related notions: $(L_0,L_1,0)$-smoothness, which is empirically verified  \cite{zhang2019gradient} for neural networks training and is theoretically proved for phase retrieval from \cite{chen2023generalized} and the appendix, and $(L_0,0,L_2)$-smoothness, which is theoretically proven for specific shallow neural networks from \cite{taheri2023fast} and the appendix. 
 	%The $(L_0,L_1,L_2)$-smoothness condition, thus subsumes both and provides a unified framework for analysis.

%We show in Remark that a two-layer  neural network satisfies Definition \ref{definition 1} under some regularity conditions but out of the range of Lipschitz smoothness and $\left(L_0,L_1\right)$- smoothness. 

%Recently, AGD and its variants have become popular especially for solving large scale machine learning problems, since they can accelerate the training rate with only first-order information.
Our analysis relies on a relaxed affine-variance noise condition, which will be formally defined in (\ref{noise assumption}) \cite{hongrevisiting,yu2025convergence}. 
This condition was initially proposed by \cite{khaledbetter} in the expected form given in (\ref{expectation noise assumption}), and many practical stochastic gradient settings, such as sub-sampling and compression schemes satisfy this noise model, but not bounded variance or the strong growth condition that
the stochastic gradient $g(\vx)$ of $f$ at $\vx$ satisfies, for some non-negative constants $B$,
\begin{align}
	\label{strong growth condition}
	\mathbb{E}\left[\norm{g(\vx)-\nabla f(\vx)}^2\right]\leq B\norm{\nabla f(\vx)}^2, \forall \vx\in\mathbb{R}^d.
\end{align}  
Closely related to our works are 
\cite{vaswani2019fast,gupta2024nesterov}.
Under the strong growth condition,
\cite{vaswani2019fast} analyzed the Accelerated Coordinate Descent method (ACDM) \cite{nesterov2012efficiency}, while \cite{gupta2024nesterov} studied SNAG when $B\leq 1$.
Both works achieved the expected accelerated convergence rates in the (strongly) convex setting, but only under the standard smoothness condition.

%Numerous literature have payed attention to some relaxed noise assumptions, to better characterize the noise model in practice.
%Considering the following strong growth condition i.e., the stochastic gradient $g(\vx)$ of $f$ at $\vx$ satisfies that for some non-negative constants $B$,
%\begin{align}
%  \label{strong growth condition}
%  \mathbb{E}\left[\norm{g(\vx)-\nabla f(\vx)}^2\right]\leq B\norm{\nabla f(\vx)}^2, \forall \vx\in\mathbb{R}^d,
%\end{align}
%\cite{vaswani2019fast} analyzed the Accelerated Coordinate Descent (ACDM) method proposed by \cite{nesterov2012efficiency} and \cite{gupta2024nesterov} studied SNAG under $B\leq 1$.
%Both of them obtained the expected acceleration results in the (strongly) convex standard-smooth setting.

We summarize our main contributions as follows.
\begin{enumerate}
  \item \label{contribution a} Motivated by several machine learning problems, we propose a more general smoothness condition defined in Definition \ref{definition 1}.
  \item \label{contribution b} Under this new smoothness condition, we analyze NAG in the deterministic and convex setting, and we show the accelerated convergence rate of order $\mathcal{O}\left(1/T^2\right)$, matching the optimal rate in \cite{nesterov1983method}.
  \item \label{contribution c} For stochastic optimizations under this general smoothness, we focus on the sub-Gaussian version of
  relaxed affine-variance noise (Assumption \ref{assumption 3}),
  and we prove that SNAG converges at the rate of $\tilde{\mathcal{O}}\left(1/T^2+\sqrt{\left(A+B+C\right)/T}\right)$ in high probability.  This rate  matches the optimal convergence rate for smooth convex optimization under bounded variance noise \cite{lan2012optimal,ghadimi2016accelerated}.
 It could improve to $\tilde{\mathcal{O}}\left(1/T^2\right)$ if the noise parameters $A,B$ and $C$ are small enough.
  \item \label{contribution d}  As a byproduct, we apply our analysis to standard smooth optimization under the expected relaxed affine-variance noises (Assumption \ref{assumption 4}), and we  demonstrate that SNAG reaches the convergence rate of order $\mathcal{O}\left(\left(1+B\right)/T^2+\sqrt{\left(A+C\right)/T}\right)$ in expectation. 
  %This rate could also accelerate to $\mathcal{O}\left(1/T^2\right)$ when $A,B$ and $C$ are sufficiently small. 
\end{enumerate}

The rest of this paper are organized as follows. We first briefly discuss some extra works related to NAG, generalized smoothness condition and the relaxed noise assumption in the next section. 
We then introduce some necessary assumptions and notations in Section \ref{section preliminary}. 
In Section \ref{subsection 4.1} and Section \ref{subsection 4.2}, we provide the deterministic convergence result and high-probability convergence result under  $(L_0,L_1,L_2)$-smoothness respectively.
In Section \ref{section 5}, we present the expected convergence rate of SNAG under the classic smoothness.
In Section \ref{numerical experiment}, we conduct numerical experiments and show the better performance of SNAG compared to SGD for the two-layer neural network and the phase retrieval model.
We also provide the convergence result for non-convex stochastic optimization under the generalized smoothness and relaxed noise assumptions in Appendix \ref{section non-convex}.
All the omitted proofs and lemmas are in the appendix.

\section{Related work}
\label{related work}

We only briefly mention the most related works due to space and knowledge constraints.

\textbf{Accelerated Gradient Descent}
%As the pioneer of accelerated gradient descent (AGD) methods,
 NAG \cite{nesterov1983method} was originally designed for smooth and convex optimizations in the deterministic setting, and it achieved the accelerated convergence rate of order $\mathcal{O}\left(1/T^2\right)$, compared to $\mathcal{O}\left(1/T\right)$ of GD.
%From then on, 
Numerous literature focused on the theoretical and practical convergence behavior of NAG and its variants \cite{nesterov2005smooth,beck2009fast}.
For example, \cite{su2016differential}  introduced a second-order ODE and accompanying tools for characterizing NAG.
\cite{lan2012optimal} generalized NAG for non-smooth and stochastic convex problems under certain conditions and provided optimal convergence rates under proper step sizes.
\cite{ghadimi2016accelerated} proposed RSAG\footnote{Indeed, RSAG in \cite{ghadimi2016accelerated} and SNAG in this paper have the same form.}, and showed  expected convergence rate of $\mathcal{O}\left(1/T+C/\sqrt{T}\right)$ in the non-convex case while $\mathcal{O}\left(1/T^2+C/\sqrt{T}\right)$ in the convex case, both under bounded variance noises and smoothness. 
\cite{li2024convex} obtained convergence rate of order $\mathcal{O}\left(1/T^2\right)$ for NAG under generalized smoothness and convexity, matching those for standard smooth convex optimizations. 
Their analysis is limited to the non-stochastic case.
Under mild noises in (\ref{strong growth condition}) and standard smoothness, \cite{vaswani2019fast} proved that ACDM \cite{nesterov2012efficiency}, which is a variant of SNAG, could reach expected accelerated convergence rates in both convex and strongly convex cases.
%Along this direction,
Under the same setting, \cite{gupta2024nesterov} proposed a new accelerated gradient method named AGNES  and they proved that the algorithm could achieve acceleration, requiring fewer hyperparameters than ACDM.
Furthermore, \cite{hermant2025gradient} showed the expected convergence rate of $\mathcal{O}\left(\left(B+1\right)/T^2\right)$ and almost-sure rate of $o\left(\left(B+1\right)/T^2\right)$ for ACDM in general convex optimization problems, and they derived fast convergence rates for ACDM in strongly convex optimization problems.
\iffalse
Under the noise assumption of $\mathbb{E}\left[\norm{g(\vx)}\right]\leq \tilde{B}\norm{\nabla f(\vx)}^2, \forall \vx\in\mathbb{R}^d$
 and standard smoothness,  \cite{hermant2025gradient} showed  convergence rates of $\mathcal{O}\left(\tilde{B}/T^2\right)$ for SNAG
 in general convex optimization problems, and they also derived fast convergence rates for SNAG in strongly convex optimization problems.
\fi

\begin{table}[h]
\small
\renewcommand \arraystretch{1.8} 
\caption{Related works under the generalized smoothness condition.}
\label{table 1}
\vspace{.1 in}
\begin{center}
\begin{threeparttable}
  \begin{tabular}{|c|c|c|c|c|c|c|}
\hline
  &\textbf{Alg.} & \textbf{Convexity} & \textbf{Noise} & \textbf{Smoothness}& \makecell{\textbf{Conv.} \\ \textbf{type}} & \makecell{\textbf{Conv.} \\ \textbf{rate}} \\
\hline
\cite{zhang2019gradient} & SGD & non-convex & \makecell{bounded \\(a.s.)} & $(L_0,L_1)$ &  $\mathbb{E}$ & $\frac{1+C}{\sqrt{T}}$\\
\hline
\cite{li2024convex} & SGD & non-convex & \makecell{bounded \\ variance}  & generalized $(L_0,L_1)$ &  $\mathbb{E}$ &$\frac{1+\sqrt{C}}{\sqrt{T}}$  \\
\hline
\cite{li2024convex} & NAG & convex & - & generalized $(L_0,L_1)$ &  - & $\frac{1}{T^2}$ \\
\hline
\multirow{2}{*}{\cite{yu2025convergence}} & \multirow{2}{*}{\makecell{SGD or\\ SNAG}} &  non-convex &  \multirow{2}{*}{\makecell{relaxed affine \\ (a.s.)}} & \multirow{2}{*}{$(L_0,L_1)$} &  \multirow{2}{*}{\textbf{w.h.p}} & \multirow{2}{*}{$\frac{1}{T}+\frac{\sqrt{A}+\sqrt{C}}{\sqrt{T}}$} \\ \cline{3-3} 
 &  & convex & & & & \\
\hline
Thm \ref{theorem 1} & NAG & convex & - & $(L_0,L_1,L_2)$ &  - & $\frac{1}{T^2}$ \\
\hline
Thm \ref{theorem 2} & SNAG & convex & \makecell{relaxed affine \\(sub-Gaussian form)} & $(L_0,L_1,L_2)$ & \textbf{w.h.p} & $\frac{1}{T^2}+\sqrt{\frac{A+B+C}{T}}$  \\
\hline
Thm \ref{theorem 3} & SNAG & convex & relaxed affine  & Lipschitz & $\mathbb{E}$ & $\frac{1}{T^2}+\sqrt{\frac{A+C}{T}}$ \\
\hline
Thm \ref{theorem 4} & SNAG & non-convex & \makecell{relaxed affine \\(sub-Gaussian form)} & $(L_0,L_1,L_2)$ & \textbf{w.h.p} & $\frac{1}{T}+\sqrt{\frac{A+C}{T}}$\\
\hline
\end{tabular}
\begin{tablenotes}
\scriptsize
\item[1] Indeed, \cite{li2024convex} provided the probabilistic results for SGD while the dependence of the probability margin is the polynomial of $1/\delta$. In order to distinguish them from other high-probability results with dependence of $\log\frac{T}{\delta}$, we consider them as the expected results.
\item[2] ``Alg.'', ``con.'' and ``cond.'' are the shorthand of the words ``Algorithm'', ``convergence'' and ``condition''. 
\item[3] We ignore the dependence on the noise parameters order and the logarithm factors of the horizon $T$ in this table.
\item[4] The extra condition for the gradient is required in \cite{zhang2019gradient}.
\end{tablenotes}
\end{threeparttable}
\end{center}
\end{table}

\begin{table}[!h]
\small
\renewcommand \arraystretch{1.8} 
\caption{Previous works related to NAG.}
\label{table 2}
\vspace{.1 in}
\begin{center}
\begin{threeparttable}
  \begin{tabular}{|c|c|c|c|c|c|c|}
\hline
  &\textbf{Agl.} & \textbf{Convexity} & \textbf{Noise} & \textbf{Smoothness} & \makecell{\textbf{Conv.} \\ \textbf{type}} & \makecell{\textbf{Conv.} \\ \textbf{rate}}  \\ 
\hline
\cite{nesterov1983method} & NAG & convex & - & Lipschitz & - & $\frac{1}{T^2}$ \\
\hline
\multirow{2}{*}{\cite{ghadimi2016accelerated}} & \multirow{2}{*}{SNAG} & non-convex & \multirow{2}{*}{\makecell{bounded \\ variance}} & \multirow{2}{*}{Lipschitz} & \multirow{2}{*}{$\mathbb{E}$} & $\frac{1}{T}+\sqrt{\frac{C}{T}}$  \\ \cline{3-3} \cline{7-7}
 &  & convex &  & & & $\frac{1}{T^2}+\sqrt{\frac{C}{T}}$ \\
\hline
\cite{vaswani2019fast} & SNAG & convex & strongly growth & Lipschitz & $\mathbb{E}$ & $\frac{B+1}{T^2}$ \\
\hline
\cite{li2024convex} & NAG & convex & -  & generalized $(L_0,L_1)$ & - & $\frac{1}{T^2}$ \\
\hline
\multirow{2}{*}{\cite{hermant2025gradient}} & \multirow{2}{*}{SNAG} & \multirow{2}{*}{convex} & \multirow{2}{*}{strongly growth} & \multirow{2}{*}{Lipschitz} & $\mathbb{E}$ & \multirow{2}{*}{$\frac{B+1}{T^2}$}  \\  \cline{6-6} 
 &  &  &  &  & a.s. & \\
\hline
Thm \ref{theorem 1} & NAG & convex & - & $(L_0,L_1,L_2)$ &  - & $\frac{1}{T^2}$  \\
\hline
Thm \ref{theorem 2} & SNAG & convex & \makecell{relaxed affine \\(sub-Gaussian form)} & $(L_0,L_1,L_2)$ & \textbf{w.h.p} & $\frac{1}{T^2}+\sqrt{\frac{A+B+C}{T}}$   \\
\hline
Thm \ref{theorem 3} & SNAG & convex & relaxed affine  & Lipschitz & $\mathbb{E}$ & $\frac{1}{T^2}+\sqrt{\frac{A+C}{T}}$  \\
\hline
Thm \ref{theorem 4} & SNAG & non-convex & \makecell{relaxed affine \\(sub-Gaussian form)} & $(L_0,L_1,L_2)$ & \textbf{w.h.p} & $\frac{1}{T}+\sqrt{\frac{A+C}{T}}$\\
\hline
\end{tabular}
\begin{tablenotes}
\scriptsize
\item[1] As discussed in Section \ref{related work}, \cite{vaswani2019fast,hermant2025gradient} analyzed ACDM, which is a variant of SNAG. However, ACDM is equivalent to SNAG with the specific step size setting in the convex case.
\item[2] ``Alg.'', ``con.'' and ``cond.'' are the shorthand of the words ``Algorithm'', ``convergence'' and ``condition''. 
\item[3] We ignore the dependence on the noise parameters order and the logarithm factors of the horizon $T$ in this table.
\end{tablenotes}
\end{threeparttable}
\end{center}
\end{table}

\textbf{Relaxed affine-variance noise and its variants}
Affine-variance noise (i.e., $A=0$ in (\ref{expectation noise assumption})) has attracted increasing attention as it can characterize gradient noises in many practical problems, such as machine learning with feature noise \cite{fuller2009measurement,khani2020feature}, robust linear regression \cite{xu2008robust} and multilayer networks \cite{faw2022power}.
\cite{bottou2018optimization} analyzed vanilla SGD in expectation and pointed out that there is no essential difference in the analysis between the bounded variance noise and the affine-variance noise under standard smoothness.
For Adagrad-Norm, \cite{faw2022power} provided expected convergence rates of order $\tilde{\mathcal{O}}\left(1/\sqrt{T}\right)$ in the non-convex setting and this rate could reach $\tilde{\mathcal{O}}\left(1/T\right)$ when $B,C$ are of order $\mathcal{O}\left(1/\sqrt{T}\right)$.
Under the same setting, \cite{wang2023convergence} further proposed a novel auxiliary function for analysis and obtained a tighter bound especially when $C=0$. 
\cite{attia2023sgd} derived high probability convergence for Adagrad-Norm in both convex and non-convex cases, under almost-sure version of affine-variance noises.
\cite{khaledbetter} proposed the relaxed affine-variance noise (see (\ref{expectation noise assumption})), and they derived an expected convergence rate of order $\mathcal{O}\left(1/\sqrt{T}\right)$ for SGD in the non-convex and smooth setting.
\cite{hongrevisiting} considered sub-Gaussian version of the relaxed affine-variance noise, and they derived probabilistic convergence rates under $\left(L_0,L_1\right)$-smoothness. \cite{yu2025convergence} analyzed RSAG (covering SGD as a special case) in both convex and non-convex settings under $\left(L_0,L_1\right)$-smoothness.

\textbf{Generalized smoothness}
%Since global smoothness is limited in many practical problems,
Motivated by practical observations,
 \cite{zhang2019gradient} proposed $\left(L_0,L_1\right)$-smoothness for twice differentiable functions.  They showed $\mathcal{O}\left(1/T\right)$ convergence rate for GD and $\mathcal{O}\left(1/\sqrt{T}\right)$ convergence rate for SGD in the non-convex setting, involving extra clipping mechanisms.
\cite{zhang2020improved} improved the convergence analysis on problem-dependent parameters for clipped SGD under essentially the same smoothness.
In the analysis of Adagrad-Norm under affine-variance noises, \cite{faw2023beyond} derived convergence bounds of order $\tilde{\mathcal{O}}\left(1/\sqrt{T}\right)$ in the non-convex case when $B<1$.
 \cite{wang2023convergence} gave a counter-example showing the necessity of prior knowledge on problem parameters for learning rates in AdaGrad under $\left(L_0,L_1\right)$-smoothness. 
Via a notion of continuity, \cite{guille2021study}  demonstrated that the
 strong convexity and smoothness have a weakness resulting in a lack of
 robustness for tuning first-order algorithms, and they presented promising alternatives.
 
We list the most relevant works in Table \ref{table 1} and Table \ref{table 2} for comparisons.

% presented an argument on the necessity to adopt different conditions from the ones
% classically used (smoothness and strong convexity), in order to tune the hyperparameters of first-order optimization in a
% meaningful way. 
 
%Inspired by the research about two-layer neural network \cite{taheri2023fast} and the phase retrial \cite{jin2021non}, we further general $\left(L_0,L_1\right)$-smoothness condition as defined in Definition \ref{definition 1}.
% \cite{li2024convergence,hongrevisiting,hong2024convergence,yu2025convergence} studied adaptive-type algorithms, and derived high-probability convergence rate under $\left(L_0,L_1\right)$-smoothness.

\section{Preliminaries}
\label{section preliminary}
We consider Problem (\ref{problem}) over the Euclidean space $\mathbb{R}^d$ with the $l_2$ norm, denoted as $\norm{\cdot}$.
We first introduce the following assumption.
\begin{assumption}[Below bounded]
   \label{assumption 1} There exists a minimizer $\vx^*\in \mathbb{R}^d$ and the objective function is bounded from below, i.e.,  $$f\left(\vx^*\right)=f^*:=\inf_{\vx\in\mathbb{R}^d}f(\vx)>-\infty.$$
\end{assumption}
In the stochastic setting, we make the following assumptions.
\begin{assumption}[Unbiased estimator]
   \label{assumption 2} The gradient oracle returns an unbiased estimator of $\nabla f(\vx)$, i.e., for all $\vx\in\R^d$, 
  \begin{align*}
    \mathbb{E}_\rvz\left[\nabla f_\rvz(\vx;\rvz)\right]=\nabla f(\vx).
  \end{align*}
\end{assumption}
 \begin{assumption}[Relaxed affine-variance (sub-Gaussian form)]
   \label{assumption 3} The gradient oracle satisfies that for some constants $A,B,C\geq 0$, 
  \begin{align}
     \label{noise assumption}
  \mathbb{E}_{\rvz} \left[\exp\left(\frac{\norm{\nabla f_\rvz(\vx;\rvz)-\nabla f(\vx)}^2}{A\left(f(\vx)-f^*\right)+B\norm{\nabla f\left(\vx\right)}^2+C}\right)\right]  \leq \exp(1), \forall \vx\in\mathbb{R}^d.
  \end{align}
 \end{assumption}
 \begin{assumption}[Relaxed affine-variance (expected form)]
  \label{assumption 4} The gradient oracle satisfies that for some constants $A,B,C\geq 0$, 
\begin{align}
  \label{expectation noise assumption}
    \mathbb{E}_{\rvz}\left[\norm{\nabla f_\rvz(\vx;\rvz)-\nabla f(\vx)}^2\right] \leq A\left(f(\vx)-f^*\right)+B\norm{\nabla f\left(\vx\right)}^2+C, \forall \vx\in\mathbb{R}^d.
\end{align}
\end{assumption}
Assumption \ref{assumption 2} is a relevant assumption for studying many practical settings and is also commonly used in the analysis of stochastic optimization.
Assumption \ref{assumption 3} is weaker than the bounded noise in \cite{zhang2019gradient,zhang2020improved} and the almost-sure version of (relaxed) affine-variance noise in \cite{attia2023sgd,hongrevisiting,yu2025convergence}. 
Although Assumption \ref{assumption 3} is stronger than its expected version in Assumption \ref{assumption 4} as it controls all moments of the noise distribution, while Assumption 4 only controls its second moment (the variance), the former one could lead to high-probability convergence, which could ensure corresponding expected convergences.
Assumption \ref{assumption 4} was initially proposed by \cite{khaledbetter} under the name expected smoothness. Its original, equivalent form is:
  $\mathbb{E}_{\rvz}\left[\norm{\nabla f_\rvz(\vx;\rvz)}^2\right]\leq A\left(f(\vx)-f^*\right)+\tilde{B}\norm{\nabla f\left(\vx\right)}^2+C, \forall \vx\in\mathbb{R}^d$.

\textbf{Notations}
We denote the set $\{1,\cdots,T\}$ as $[T]$. 
We use $\mathbb{E}_t[\cdot]\triangleq\mathbb{E}[\cdot|\rvz_1,\cdots,\rvz_{t-1}]$ to represent the conditional expectation, where $\rvz_i$ is the random sample in the $i$-th gradient oracle. 
The notation $a\sim \mathcal{O}(b)$ and $a\leq \mathcal{O}(b)$ refer to $c_1 b\leq a\leq c_2 b$ and $a\leq c_3 b$ with $c_1, c_2, c_3$ being positive constants, respectively. 
Also, we write $\tilde{\mathcal{O}}(b)$ for $\mathcal{O}(b \cdot \text{poly}(\log b))$. Throughout the paper, we define $0^0=1$.

\section{Convergence of NAG under $(L_0,L_1,L_2)$-smoothness for deterministic optimization}
\label{subsection 4.1}
We first present NAG algorithm and its convergence rate in the deterministic case.
The step size setting in Algorithm \ref{algorithm1} slightly differs from that of original NAG \cite{nesterov1983method}, in which the latter directly sets $A_t=B_t$.
%Such a modification could be used to control the gradient norms (or function value gaps) in the analysis.
Such a modification is proposed by \cite{li2024convex} where they obtained the optimal convergence rate under a general smoothness for convex non-stochastic optimizations.
\begin{algorithm}[H]
  \caption{Nesterov's Accelerated Gradient Descent (NAG)}
  \label{algorithm1}
  \begin{algorithmic}[1]
  \State {\bfseries Input:} Horizon $T$, $\xag{0}=\vx_0\in \mathbb{R}^d$, step sizes $\beta, \{\lambda_t\}_{t\in [T]}$ and $A_0=1/\beta, B_0=0$.
  \For{$t=1,\cdots, T$}
  \State $B_t=B_{t-1}+\frac{1}{2}\left(1+\sqrt{4B_{t-1}+1}\right)$;
  \State $A_t=B_t+\frac{1}{\beta}$;
  \State $\xmd{t}=\frac{A_{t-1}}{A_t}\xag{t-1}+\left(1-\frac{A_{t-1}}{A_t}\right)\vx_{t-1}$;
  \State $\vx_t=\vx_{t-1}-\lambda_t \nablaf{t}{md}$;
  \State $\xag{t}=\xmd{t}-\beta\nablaf{t}{md}$.
  \EndFor
  \end{algorithmic}
\end{algorithm}
To better understand the NAG method, we provide the following lemma summarized from \cite{d2021acceleration,li2024convex}.
\begin{lemma}
  \label{lemma 2.1}
  For all $0\leq t\leq T$, we have 
  \begin{enumerate}
    \item $\frac{1}{4}t^2\leq B_t\leq t^2$;
    \item $\left(A_t-A_{t-1}\right)^2=\left(B_t-B_{t-1}\right)^2=B_t < A_t$;
    \item $A_t-A_{t-1}=B_t-B_{t-1}\geq 1$. Thus, $\left\{A_t\right\}_{t\in[T]}$ and $\left\{B_t\right\}_{t\in[T]}$ are both monotonically increasing sequences.
  \end{enumerate}
\end{lemma}
The above lemma plays vital roles both in the induction argument for bounding the function value gap and in the final convergence analysis.
Refer to Appendix \ref{section omitted proof} for the complete proof.
\subsection{Convergence results}
In this subsection, we provide the theoretical convergence behaviour of Algorithm \ref{algorithm1} under $(L_0,L_1,L_2)$-smoothness condition.
\begin{theorem}
  \label{theorem 1}
  Let $T>0$ and $f$ be an $\left(L_0,L_1,L_2\right)$-smooth convex function. Suppose that$\left\{\xag{t}\right\}_{t\in[T]}$ is a sequence generated by \Algref{algorithm1} with step sizes $\beta, \lambda_t$ satisfying
  \begin{align}
  	 \label{betaAlg1}
    \beta=\frac{1}{\LFone}, \qquad \lambda_t=\left(A_t-A_{t-1}\right)\beta,
  \end{align}
  where $\LFone$ is a constant, depending on the smoothness parameters $\left\{L_i\right\}_{i\in[3]}, p, q$, with its explicit expression in (\ref{L1}).
  Then, under Assumption \ref{assumption 1}, we have\footnote{We state the explicit convergence result in (\ref{Complete theorem 1}).}
  \begin{align}
    \fxag{T}-f^*\leq \mathcal{O}\left(\frac{1}{T^2}\right).
  \end{align}
\end{theorem}
Considering the definition of $\LFone$ in (\ref{L1}), $\beta$ could reduce to $1/2L$ when the objective function is $L$-smooth, which aligns with $\beta=1/L$ in \cite{nesterov1983method} up to a constant.
Furthermore, Theorem \ref{theorem 1} recovers the convergence rate of order $\mathcal{O}\left(1/T^2\right)$ obtained in \cite{nesterov1983method} where the smoothness is required.
This bound is optimal \cite{nemirovskij1983problem} for smooth convex optimization when $d$ is large enough.
%Since the new generalized smoothness in Definition \ref{definition 1} is weaker than the global smoothness, we believe that the convergence rate in \ref{theorem 1} cannot be improved.
\subsection{Proof of Theorem \ref{theorem 1}}
\label{proof of theorem 1}
We first present the explicit expressions of $\mathcal{C}_1$, $\Fone$, $\LFone$,
  \begin{align}
      \label{C1}
      \mathcal{C}_1=\triangleag{0}+\frac{1}{2}\norm{\vx_0-\vx^*}^2,
    \end{align}
    \begin{align}
    \label{Fone}
    \Fone=\mathcal{C}_1+ \frac{1}{L_1+L_2}\sqrt{g(\mathcal{C}_1)}+\frac{L_0+L_1\left(g(\mathcal{C}_1)\right)^{\frac{p}{2}}+L_2\mathcal{C}_1^q}{2\left(L_1+L_2\right)^2},
  \end{align}
    \begin{align}
    \label{L1}
    \LFone=\Ldeterministic,
  \end{align}
  where $g\left(\cdot\right): \mathbb{R}\rightarrow\mathbb{R}$ is a function defined in (\ref{g function}).
Then, we will provide some useful lemmas. 
Lemma \ref{lemma B.1} follows from the analysis in \cite{ghadimi2016accelerated} and is derived from the iteration steps in Algorithm \ref{algorithm1}.
\begin{lemma}
  \label{lemma B.1}
  Let $\left\{\xmd{k}\right\}_{k\in[T]}$ and $\left\{\xag{k}\right\}_{k\in[T]}$ be the two sequences generated by Algorithm \ref{algorithm1}. Then we have for all $k\in[T]$,
  \begin{align*}
      \norm{\xmd{k}-\xag{k-1}}^2\leq \frac{1}{A_k\cdot A_{k-1}}\sumlimits{i=1}{k-1} \frac{A_i^2\cdot \left(\lambda_i-\beta\right)^2}{A_i-A_{i-1}}\norm{\nablaf{i}{md}}^2.
  \end{align*}
\end{lemma}
\begin{proof}
  From Algorithm \ref{algorithm1}, we have
  \begin{align*}
    \xag{k}-\vx_{k}= & \xmd{k}-\beta\nablaf{k}{md}-\vx_{k-1}+\lambda_k\nablaf{k}{md}\notag\\
    = & \frac{A_{k-1}}{A_k}\left(\xag{k-1}-\vx_{k-1}\right)+\left(\lambda_k-\beta\right)\nablaf{k}{md}.
  \end{align*}
Since $\xag{0}=\vx_0$, we obtain that
\begin{align*}
  \xag{k}-\vx_{k}= & \frac{1}{A_k}\sumlimits{i=1}{k}A_i\left(\lambda_i-\beta\right)\nablaf{i}{md},
\end{align*}
which implies
\begin{align}
  \label{eq 16.1}
  \norm{\xag{k}-\vx_{k}} \leq & \frac{1}{A_k}\sumlimits{i=1}{k}A_i\left|\lambda_i-\beta\right| \cdot \norm{\nablaf{i}{md}}.
\end{align}
Applying the iteration step in Algorithm \ref{algorithm1} again, we have
 \begin{align*}
  \xmd{k}-\xag{k-1}=\left(1-\frac{A_{k-1}}{A_k}\right)\left(\vx_{k-1}-\xag{k-1}\right).
 \end{align*}
Combining with (\ref{eq 16.1}), we have
  \begin{align}
      \label{1.4}
      \norm{\xmd{k}-\xag{k-1}}\leq \left(1-\frac{A_{k-1}}{A_k}\right) \cdot \frac{1}{A_{k-1}}\sumlimits{i=1}{k-1} A_i\left|\lambda_i-\beta\right| \cdot \norm{\nablaf{i}{md}}.
  \end{align}
 Using the fact that 
  \begin{align*}
    \sumlimits{i=1}{k-1} A_i \cdot \left(1-\frac{A_{i-1}}{A_i}\right) = A_{k-1}-A_0,
  \end{align*}
  we have
  \begin{align*}
    & \norm{\xmd{k}-\xag{k-1}}^2\notag\\
    \leq &\left(1-\frac{A_{k-1}}{A_k}\right)^2 \cdot \frac{1}{A_{k-1}^2}\left[\sumlimits{i=1}{k-1}A_i\left(1-\frac{A_{i-1}}{A_i}\right) \cdot \frac{\left|\lambda_i-\beta\right|}{1-\frac{A_{i-1}}{A_i}} \norm{\nablaf{i}{md}}\right]^2\notag\\
    \leq & \left(1-\frac{A_{k-1}}{A_k}\right)^2  \cdot \frac{A_{k-1}-A_0}{A_{k-1}^2}\sumlimits{i=1}{k-1}A_i\left(1-\frac{A_{i-1}}{A_i}\right) \frac{A_i^2\cdot \left(\lambda_i-\beta\right)^2}{\left(A_i-A_{i-1}\right)^2}\norm{\nablaf{i}{md}}^2\notag\\
    \leq & \frac{1}{A_k\cdot A_{k-1}}\sumlimits{i=1}{k-1} \frac{A_i^2\cdot \left(\lambda_i-\beta\right)^2}{A_i-A_{i-1}}\norm{\nablaf{i}{md}}^2,
  \end{align*}
  where the second inequality follows from Jensen's inequality and the last inequality holds since Lemma \ref{lemma 2.1}.
\end{proof}
In the next Lemma, we assume the function value gap is bounded. Under this assumption, the analysis for $\left(L_0,L_1,L_2\right)$-smooth and $L$ smooth objective functions becomes similar. 
With reference to \cite{nesterov1983method,ghadimi2016accelerated,d2021acceleration}, we provide the following analysis with a step size specific to the novel smoothness condition.
\begin{lemma}
  \label{lemma C.2}
  Suppose that $\fxmd{l}-f^*\leq \Fone, \forall l\in[t]$. 
  Then, under the conditions of Theorem \ref{theorem 1}, we have
    \begin{align*}
       \triangleag{t}\leq \frac{\mathcal{C}_1}{A_t\beta},
    \end{align*}
    where  $\mathcal{C}_1$ is a constant related to the initial point and is defined in (\ref{C1}).
\end{lemma}
\begin{proof}
  By Corollary \ref{corollary 1} and the assumption that $\trianglemd{l}\leq \Fone, \forall l\in[t]$, we have $\norm{\nablaf{l}{md}}^2\leq g(\Fone), \forall l\in[t]$.
    Therefore,
    \begin{align*}
        \norm{\vx_l^{ag}-\vx_l^{md}}=\beta \norm{\nablaf{l}{md}}\leq \beta \sqrt{g(\Fone)}\leq \min\left\{1/L_1,1/L_2\right\},
    \end{align*}
    where the last inequality holds since $\beta=\frac{1}{\mathcal{L}_1}$ with $\mathcal{L}_1$ defined in (\ref{L1}).
    By Lemma \ref{lemma A.2}, we have 
    \begin{align}
      \label{eq 5}
       \fxag{l}\leq & \fxmd{l}+\la \nablaf{l}{md},\xag{l}-\xmd{l}\ra+\frac{L_0+L_1 \left(g\left(\Fone\right)\right)^\frac{p}{2}+L_2 \Fone^q}{2}\norm{\xag{l}-\xmd{l}}^2\notag\\
        = & \fxmd{l}-\beta\norm{\nablaf{l}{md}}^2+\frac{L_0+L_1 \left(g\left(\Fone\right)\right)^\frac{p}{2}+L_2 \Fone^q}{2}\beta^2\norm{\nablaf{l}{md}}^2.
    \end{align}
    By the convexity and the iteration step in Algorithm \ref{algorithm1}, we have
    \begin{align}
        \label{1.8}
        &\fxmd{l}-\left[\frac{A_{l-1}}{A_l}\cdot \fxag{l-1}+\left(1-\frac{A_{l-1}}{A_l}\right)\cdot f^*\right]\notag\\
        = & \left(1-\frac{A_{l-1}}{A_l}\right)\cdot \left[\fxmd{l}-f^*\right]+\frac{A_{l-1}}{A_l}\cdot\left[\fxmd{l}-\fxag{l-1}\right]\notag\\
        \leq & \left(1-\frac{A_{l-1}}{A_l}\right)\cdot \la \nablaf{l}{md}, \vx_l^{md}-\vx^*\ra +\frac{A_{l-1}}{A_l}\cdot\la \nablaf{l}{md}, \vx_l^{md}-\vx_{l-1}^{ag}\ra\notag\\
        = & \la \nablaf{l}{md}, \left(1-\frac{A_{l-1}}{A_l}\right)\cdot\left(\vx_l^{md}-\vx^*\right)+\frac{A_{l-1}}{A_l}\cdot\left(\vx_l^{md}-\vx_{l-1}^{ag}\right)\ra\notag\\
        = & \left(1-\frac{A_{l-1}}{A_l}\right)\cdot \la \nablaf{l}{md},\vx_{l-1}-\vx^*\ra. 
    \end{align}
    Combining (\ref{eq 5}) and (\ref{1.8}), we obtain that
    \begin{align}
      \label{eq 7}
        \fxag{l} \leq & \frac{A_{l-1}}{A_l}\cdot \fxag{l-1}+\left(1-\frac{A_{l-1}}{A_l}\right)\cdot f^*+\left(1-\frac{A_{l-1}}{A_l}\right)\cdot \la \nablaf{l}{md},\vx_{l-1}-\vx^*\ra\notag\\
        &-\beta\norm{\nablaf{l}{md}}^2+\frac{L_0+L_1 \left(g\left(\Fone\right)\right)^\frac{p}{2}+L_2 \Fone^q}{2}\beta^2\norm{\nablaf{l}{md}}^2.
    \end{align}
    Also, 
    \begin{align*}
        &\norm{\vx_{l-1}-\vx^*}^2-2\lambda_l \la \nablaf{l}{md},\vx_{l-1}-\vx^*\ra+\lambda_l^2 \norm{\nablaf{l}{md}}^2\notag\\
        =&\norm{\vx_{l-1}-\lambda_l \nablaf{l}{md}-\vx^*}^2=\norm{\vx_l-\vx^*}^2.
    \end{align*}
    Hence, 
    \begin{align}
        \label{eq 9}
        \la \nablaf{l}{md},\vx_{l-1}-\vx^*\ra=\frac{1}{2\lambda_l}\left[\norm{\vx_{l-1}-\vx^*}^2-\norm{\vx_l-\vx^*}^2\right]+\frac{\lambda_l}{2}\norm{\nablaf{l}{md}}^2.
    \end{align}
    Substituting (\ref{eq 9}) into (\ref{eq 7}), we have
    \begin{align}
      \label{eq 10}
        \fxag{l}\leq &\frac{A_{l-1}}{A_l}\cdot \fxag{l-1}+\left(1-\frac{A_{l-1}}{A_l}\right)\cdot f^*+\frac{A_l-A_{l-1}}{2A_l\cdot\lambda_l}\left[\norm{\vx_{l-1}-\vx^*}^2-\norm{\vx_l-\vx^*}^2\right]\notag\\
        & -\beta\left(1-\frac{\left(L_0+L_1 \left(g\left(\Fone\right)\right)^\frac{p}{2}+L_2 \Fone^q\right)\beta}{2}-\frac{\lambda_l\left(A_l-A_{l-1}\right)}{2\beta A_l}\right)\norm{\nablaf{l}{md}}^2.
    \end{align}
    By the constraint of $\lambda_l$ in (\ref{betaAlg1}) and applying Lemma \ref{lemma 2.1}, we obtain that
    \begin{align*}
      \lambda_l\cdot \frac{A_l-A_{l-1}}{A_l}=\beta\cdot \frac{\left(A_l-A_{l-1}\right)^2}{A_l}=\beta\frac{\left(B_l-B_{l-1}\right)^2}{B_{l}+1/\beta}=\beta\frac{B_l}{B_l+1/\beta}< \beta.
    \end{align*} 
   Also, recalling the constraint of $\beta$ in (\ref{betaAlg1}), we have
    \begin{align*}
      \left(L_0+L_1 \left(g\left(\Fone\right)\right)^\frac{p}{2}+L_2 \Fone^q\right)\beta\leq \frac{1}{2}.
    \end{align*}
    Therefore,
    \begin{align*}
      1-\frac{\left(L_0+L_1 \left(g\left(\Fone\right)\right)^\frac{p}{2}+L_2 \Fone^q\right)\beta}{2}-\frac{\lambda_l\left(A_l-A_{l-1}\right)}{2\beta A_l}\geq \frac{1}{4}.
    \end{align*}
    Combining with (\ref{eq 10}) and reorganizing the terms, we have
    \begin{align}
         \triangleag{l} \leq & -\frac{1}{4}\beta \norm{\nablaf{l}{md}}^2 + \frac{A_{l-1}}{A_l}\cdot\triangleag{l-1}+\frac{A_l-A_{l-1}}{2A_l\cdot\left(A_l-A_{l-1}\right)\beta}\left[\norm{\vx_{l-1}-\vx^*}^2-\norm{\vx_l-\vx^*}^2\right]\notag\\
        = & -\frac{1}{4}\beta \norm{\nablaf{l}{md}}^2+\frac{A_{l-1}}{A_l}\cdot\triangleag{l-1}+\frac{1}{2\beta A_l}\left[\norm{\vx_{l-1}-\vx^*}^2-\norm{\vx_l-\vx^*}^2\right].
    \end{align}
    With both sides of the above inequality multiplying $A_l$, we have
    \begin{align}
      \label{eq 13}
        A_l \triangleag{l}+\frac{1}{2\beta}\norm{\vx_l-\vx^*}^2\leq A_{l-1}\triangleag{l-1}+\frac{1}{2\beta}\norm{\vx_{l-1}-\vx^*}^2-\frac{1}{4}\beta A_l \norm{\nablaf{l}{md}}^2.
    \end{align}
    Summing up over $l\in[t]$ and re-arranging the inequality, we obtain that 
    \begin{align}
        \label{5.18}
        \triangleag{t}\leq & \frac{A_0}{A_t}\triangleag{0}+\frac{1}{2\beta A_t}\norm{\vx_0-\vx^*}^2
        =  \frac{1}{A_t\beta}\triangleag{0}+\frac{1}{2\beta A_t}\norm{\vx_0-\vx^*}^2\notag\\
        = & \frac{1}{A_t\beta}\left(\triangleag{0}+\frac{1}{2}\norm{\vx_0-\vx^*}^2\right).
    \end{align}
\end{proof}
Based on the proof for Lemma \ref{lemma C.2}, we will prove the bound of $\fxmd{t}-f^*$ for all $t\in[T]$, using an induction argument.
\begin{lemma}
  \label{lemma C.3.1}
    Under the conditions of Theorem \ref{theorem 1}, we have
    \begin{align*}
      \fxmd{t}-f^*\leq \Fone, \quad \forall t\in[T],
    \end{align*}
    where $\Fone$ is defined in (\ref{Fone}).
\end{lemma}
\begin{proof}
    It is apparent that $\fxmd{1}-f^*=\fxag{0}-f^*\leq \Fone$ since $\xag{0}=\vx_0$.
    Suppose that for some $t\in[T]$, 
    $$\fxmd{l}-f^*\leq \Fone, \forall l\in[t].$$
    Next, we will bound $\fxmd{t+1}-f^*$.
    By Lemma \ref{lemma B.1}, we have 
    \begin{align*}
      \norm{\xmd{t+1}-\xag{t}}^2\leq & \frac{1}{A_{t+1}\cdot A_{t}}\sumlimits{i=1}{t} \frac{A_i^2\cdot \left(\lambda_i-\beta\right)^2}{A_i-A_{i-1}}\norm{\nablaf{i}{md}}^2\notag\\
      \leq & \sumlimits{i=1}{t}\frac{\left(A_i-A_{i-1}-1\right)^2}{A_i-A_{i-1}}\beta^2\norm{\nablaf{i}{md}}^2\notag\\
      \leq & \beta^2 \sumlimits{i=1}{t}\left(A_i-A_{i-1}\right)\norm{\nablaf{i}{md}}^2,
    \end{align*}
    where the second and the third inequalities follows from Lemma \ref{lemma 2.1}.
     Using the fact that $A_i=B_i+1/\beta, \forall i\in[T]$ and applying Lemma \ref{lemma 2.1} again, we have
       \begin{align}
      \label{eq 16}
      \norm{\xmd{t+1}-\xag{t}}^2\leq & \beta^2 \sumlimits{i=1}{t} \sqrt{A_i}\norm{\nablaf{i}{md}}^2
      \leq  \beta^\frac{5}{2}  \sumlimits{i=1}{t} A_i \norm{\nablaf{i}{md}}^2.
    \end{align}
    Since the assumption that $\fxmd{l}-f^*\leq \Fone, \forall l\in[t]$, (\ref{eq 13}) holds here for all $l\in[t]$.
    Therefore, summing up (\ref{eq 13}) over $l\in[t]$, we have
    \begin{align}
      \label{eq 17}
      \frac{1}{4}\beta\sumlimits{i=1}{t}A_i\norm{\nablaf{i}{md}}^2\leq A_0\triangleag{0}+\frac{1}{2\beta}\norm{\vx_0-\vx^*}^2.
    \end{align}
    Combining (\ref{eq 16}) and (\ref{eq 17}) and the constraint of $\beta$, we obtain that
    \begin{align*}
      \norm{\xmd{t+1}-\xag{t}}^2\leq &\sqrt{\beta}\cdot 4\left(\triangleag{0}+\frac{1}{2}\norm{\vx_0-\vx^*}^2\right)\leq \frac{1}{\left(L_1+L_2\right)^2}.
    \end{align*}
Applying Lemma \ref{lemma A.2} again, we have
\begin{align}
  \label{eq 18}
    \fxmd{t+1} \leq &\fxag{t}+\la \nablaf{t}{ag}, \xmd{t+1}-\xag{t}\ra +\frac{L_0+L_1\norm{\nablaf{t}{ag}}^p+L_2\left(\triangleag{t}\right)^q}{2}\norm{\xmd{t+1}-\xag{t}}^2\notag\\
    \leq & \fxag{t}+\norm{\nablaf{t}{ag}}\cdot \norm{\xmd{t+1}-\xag{t}}+\frac{L_0+L_1\norm{\nablaf{t}{ag}}^p+L_2\left(\triangleag{t}\right)^q}{2}\norm{\xmd{t+1}-\xag{t}}^2\notag\\
    \leq & \fxag{t}+ \frac{1}{L_1+L_2}\norm{\nablaf{t}{ag}}+\frac{L_0+L_1\norm{\nablaf{t}{ag}}^p+L_2\left(\triangleag{t}\right)^q}{2\left(L_1+L_2\right)^2},
\end{align}
where the second inequality holds since Cauchy-Schwarz inequality.
Further, considering the assumption that $\fxmd{l}-f^*\leq \Fone, \forall l\in[t]$, (\ref{5.18}) holds here.
Noting that $A_t\beta=\left(B_t+\frac{1}{\beta}\right)\cdot \beta\geq 1$, we could deduce
\begin{align}
  \label{eq 31.1}
  \triangleag{t}\leq \triangleag{0}+\frac{1}{2}\norm{\vx_0-\vx^*}^2=\mathcal{C}_1.
\end{align}
Plugging (\ref{eq 31.1}) into (\ref{eq 18}), subtracting $f^*$ from both sides and applying Corollary \ref{corollary 1}, we obtain that
\begin{align*}
  \trianglemd{t+1}\leq \mathcal{C}_1+ \frac{1}{L_1+L_2}\sqrt{g(\mathcal{C}_1)}+\frac{L_0+L_1\left(g(\mathcal{C}_1)\right)^{\frac{p}{2}}+L_2\mathcal{C}_1^q}{2\left(L_1+L_2\right)^2}= \Fone,
\end{align*}
where $g(\cdot)$ is the function defined in (\ref{g function}).
Therefore, the induction is complete and we obtain the desired result.
\end{proof}
Now we are ready to obtain the main convergence result.
\begin{proof}[Proof of Theorem \ref{theorem 1}]
  Noting that  $\trianglemd{t}\leq \Fone, \forall t\in[T]$ proved in Lemma \ref{lemma C.3.1}, we could apply Lemma \ref{lemma C.2} and obtain that
  \begin{align*}
    \triangleag{T}\leq  \frac{1}{A_T\beta}\left(\triangleag{0}+\frac{1}{2}\norm{\vx_0-\vx^*}^2\right)=\frac{\mathcal{C}_1\cdot \LFone}{A_T}.
  \end{align*}
  Applying Lemma \ref{lemma 2.1}, we obtain that 
  \begin{align}
    \label{Complete theorem 1}
    \triangleag{T}\leq & \frac{8\mathcal{C}_1\cdot\left(L_0+L_1\left(\left(g(\Fone)\right)^{\frac{1}{2}}+\left(g(\Fone)\right)^\frac{p}{2}\right)\right)}{T^2}+\frac{8\mathcal{C}_1\left(L_2\left(\left(g(\Fone)\right)^{\frac{1}{2}}+\Fone^q\right)+8\mathcal{C}_1^2\left(L_1+L_2\right)^4\right)}{T^2}.
  \end{align}
\end{proof}

\section{Convergence of SNAG under $(L_0,L_1,L_2)$-smoothness for stochastic optimization}
\label{subsection 4.2}
In this section, we provide a probabilistic convergence result for SNAG (see Algorithm \ref{algorithm2}) under the relaxed affine-variance noise assumption of its sub-Gaussian form.
Compared to Algorithm \ref{algorithm1}, the only difference is that stochastic gradients, instead of accurate gradients, are accessible.
Obviously, Lemma \ref{lemma 2.1} still holds for the stochastic case.
\begin{algorithm}[H]
  \caption{Stochastic Nesterov's Accelerated Gradient Descent (SNAG)}
  \label{algorithm2}
  \begin{algorithmic}[1]
  \State {\bfseries Input:}  Horizon $T$, $\xag{0}=\vx_0\in \mathbb{R}^d$, step sizes $\beta, \{\lambda_t\}_{t\in [T]}$ and $A_0=1/\beta, B_0=0$.
  \For{$t=1,\cdots, T$}
  \State $B_t=B_{t-1}+\frac{1}{2}\left(1+\sqrt{4B_{t-1}+1}\right)$;
  \State $A_t=B_t+\frac{1}{\beta}$;
  \State $\xmd{t}=\frac{A_{t-1}}{A_t}\xag{t-1}+\left(1-\frac{A_{t-1}}{A_t}\right)\vx_{t-1}$;
  \State \textbf{Set} $\vg_t=\nabla f_\mathbf{z}\left(\vx_t^{md};\rvz_t\right)$;
  \State $\vx_t=\vx_{t-1}-\lambda_t \vg_t$;
  \State $\xag{t}=\xmd{t}-\beta \vg_t$.
  \EndFor
  \end{algorithmic}
\end{algorithm}
\subsection{Convergence results}
\begin{theorem}
 \label{theorem 2}
  Let $T>0$ and $\delta\in(0,\frac{1}{3})$. Suppose that $\left\{\xag{t}\right\}_{t\in[T]}$ is a sequence generated by \Algref{algorithm2}, $f$ is $\left(L_0,L_1,L_2\right)$-smooth and convex, and the step sizes $\beta, \lambda_t$ satisfy that 
  \begin{align}
    \label{beta}
    \beta=\min\left\{\frac{1}{\mathcal{G}_1},\frac{1}{\mathcal{G}_2 T^{\frac{6}{5}}},\frac{1}{\mathcal{G}_3 T^\frac{2}{3}},\frac{1}{\mathcal{M} T^{\frac{3}{2}}},\frac{1}{\mathcal{M}^2 T} \right\},
  \end{align}
  \begin{align}
    \label{lambda}
    \lambda_t=\frac{1}{4} \beta \left(A_t-A_{t-1}\right),
  \end{align}
 where $\mathcal{G}_1,\mathcal{G}_2,\mathcal{G}_3$ and $\mathcal{M}$ are polynomials of $\log\frac{T}{\delta}$, depending on the noise and smoothness parameters\footnote{The explicit expressions of these notations are presented in (\ref{setting}), (\ref{setting 2}), (\ref{setting 3}) and (\ref{setting 4}).}.
  Under Assumption \ref{assumption 1}-\ref{assumption 3}, with probability at least $1-3\delta$, we have\footnote{Refer (\ref{Complete theorem 2}) for the explicit convergence result.}
   \begin{align*}
       \functionvaluegap{T}{ag} \leq \tilde{\mathcal{O}}\left(\frac{1}{T^2}+\sqrt{\frac{A+B+C}{T}}\right).
    \end{align*}
\end{theorem}

Theorem \ref{theorem 2} provides accelerated convergence rates in high probability. 
%while \cite{yu2025convergence} obtained only $\tilde{\mathcal{O}}\left(1/T+\sqrt{\left(A+C\right)/T}\right)$ for SGD and RSAG 
Up to logarithm factors, this convergence rate matches the expected convergence rate in \cite{ghadimi2016accelerated}, where they assumed bounded variance noise and standard smoothness, and it is unimprovable for smooth convex stochastic optimization \cite{lan2012optimal}.

Furthermore, the convergence rate in Theorem \ref{theorem 2} could accelerate to $\tilde{\mathcal{O}}\left(1/T^2\right)$ if the noise parameters are sufficiently small, which matches the rate for the deterministic NAG in \cite{li2024convex} under a generalized $(L_0,L_1,0)$-smoothness: $\norm{\nabla^2 f\left(\vx\right)}\leq l\left(\norm{\nabla f\left(\vx\right)}\right)$ with a sub-quadratic non-decreasing positive function $l$
	up to logarithm factors.
Note that  \cite{li2024convex} did not provide the analysis for SNAG (but they did consider SGD) and we consider the  $(L_0,L_1,L_2)$-smoothness. 
To analyze the stochastic setting, we modify the step size slightly by a constant factor and $\beta\sumlimits{i=1}{t}A_i\norm{\nablaf{i}{md}}^2$ appears in (\ref{eq 6.20}), which makes it feasible to bound $\norm{\xmd{t+1}-\xag{t}}^2$ in stochastic optimization. Combining with several new probabilistic estimates, we could finish the proof. We refer to the following subsection for the complete proof.
Our analysis for the above theorem, which relies on Assumption \ref{assumption 3}, does not apply under the weaker noise condition of Assumption \ref{assumption 4} in the generalized smoothness.

\subsection{Proof of Theorem \ref{theorem 2}}
\label{proof of theorem 2}
In this subsection, we will prove Theorem \ref{theorem 2} by induction.
We first introduce some notations used in Theorem \ref{theorem 2}, i.e.,
\begin{align}
    \label{setting}
     & \mathcal{M}= \sqrt{A \Ftwo+B g\left(\Ftwo\right)+C}, &  &\mathcal{G}_1=\max\left\{\mathcal{G}_{1,1},\mathcal{G}_{1,2},\mathcal{G}_{1,3},\mathcal{G}_{1,4}\right\},\notag\\
    & \mathcal{G}_2=\left(595\right)^{\frac{2}{5}}\left(L_1+L_2\right)^{\frac{4}{5}}\mathcal{M}^{\frac{4}{5}}\left(\logte\right)^{\frac{2}{5}},&  &\mathcal{G}_3=\left(595\right)^{\frac{2}{3}}\left(L_1+L_2\right)^{\frac{4}{3}}\mathcal{M}^{\frac{4}{3}}\left(\logte\right)^{\frac{2}{3}}, 
  \end{align}
  where 
  \begin{align}
    \label{setting 2}
          &\mathcal{G}_{1,1}=L_1\left(\sqrt{g\left(\Ftwo\right)}+\mathcal{M}\sqrt{\logte}\right),&  &\mathcal{G}_{1,2}=L_2\left(\sqrt{g\left(\Ftwo\right)}+\mathcal{M}\sqrt{\logte}\right),\notag\\
        &\mathcal{G}_{1,3}=4\left(L_0+L_1\left(g\left(\Ftwo\right)\right)^{\frac{p}{2}}+L_2\Ftwo^q\right),&  &\mathcal{G}_{1,4}=4624\left(L_1+L_2\right)^4\mathcal{C}_2^2.
  \end{align}
Furthermore, 
\begin{align}
  \label{setting 3}
  \Ftwo=\mathcal{H}+\frac{\sqrt{g(\mathcal{H})}}{L_1+L_2}+\frac{L_0+L_1\left(g(\mathcal{H})\right)^{\frac{p}{2}}+L_2\mathcal{H}^q}{2\left(L_1+L_2\right)^2},
\end{align}
with the notations
\begin{align}
  \label{setting 4}
  \mathcal{C}_2=\triangleag{0}+2\norm{\vx_0-\vx^*}^2,\quad \mathcal{H}=2\mathcal{C}_2+ 17\logte.
\end{align}
In what follows, we will present several lemmas to prepare for the convergence analysis in the stochastic setting.
\subsubsection{Preliminaries}
The following lemma bound the noise norm under Assumption \ref{assumption 3}.
\begin{lemma}
  \label{lemma C.1}
  Given $T\geq 1$, suppose that for any $t\in[T]$, $\rvup{t}=\nabla f_\mathbf{z}\left(\vx_t;\rvz_t\right)-\nabla f\left(\vx_t\right)$ satisfies Assumption \ref{assumption 3}. Then, for any given $\delta\in(0,1)$, it holds that with probability at least $1-\delta$,
  \begin{align}
    \label{eq 32}
    \norm{\rvup{t}}^2\leq \left(A\left(f(\vx_t)-f^*\right)+B\norm{\nabla f\left(\vx_t\right)}^2+C\right)\logte, \quad \forall t\in[T].
  \end{align}
\end{lemma}
\begin{proof}
    Denote $\zeta_t=\frac{\norm{\rvup{t}}^2}{A\left(f(\vx_t)-f^*\right)+B\norm{\nabla f\left(\vx_t\right)}^2+C}$, $\forall t\in[T]$, where $T$ is fixed.
    By the definition of the noise model, we have
    $$\mathbb{E}_{t}\left[\exp\left(\zeta_t\right)\right]\leq \mathrm{e}, \quad\text{thus},\quad \mathbb{E}\left[\exp\left(\zeta_t\right)\right]\leq \mathrm{e}.$$
    By Markov's inequality, for any $\beta\in\mathbb{R}$,
    \begin{align*}
    \mathbb{P}\left(\max\limits_{t\in [T]}\zeta_t\geq \beta\right)=&\mathbb{P}\left(\exp\left(\max\limits_{t\in [T]}\zeta_t\right)\geq \mathrm{e}^\beta\right)\notag\\
    \leq & \text{e}^{-\beta}\mathbb{E}\left[\exp\left(\max\limits_{t\in [T]}\zeta_t\right)\right]\leq \text{e}^{-\beta}\mathbb{E}\left[\sumlimits{t=1}{T}\exp\left(\zeta_t\right)\right]\leq \text{e}^{-\beta}T \text{e}.
    \end{align*}
    Therefore, with probability at least $1-\delta$, we have 
    \begin{align*}
    \norm{\rvup{t}}^2\leq & \left(A\left(f(\vx_t)-f^*\right)+B\norm{\nabla f\left(\vx_t\right)}^2+C\right)\logte, \quad \forall t\in[T].
    \end{align*}
\end{proof}
Next, we will establish a probabilistic upper bound for summation of the two martingale difference sequences based on the noise assumption and Lemma \ref{lemma high probability}.
We denote $\rvxi{t}=\vg_t-\nablaf{t}{md}$ in the rest of the paper.
\begin{lemma}
    \label{lemma 6.1}
    Given $T\geq 1$ and $\delta\in(0,1)$, if Assumptions \ref{assumption 1}, \ref{assumption 2} and \ref{assumption 3} hold, then \withprobability, for all $l\in[T]$, we have 
    \begin{align}
        \label{6.2}
        \sumlimits{k=1}{l} -A_k\la \rvxi{k}, \nablaf{k}{md}\ra\leq \frac{1}{4A_T\mathcal{M}^2}\sumlimits{k=1}{l} A_k^2\norm{\nablaf{k}{md}}^2\mathcal{M}_k^2+3A_T\mathcal{M}^2\log\frac{T}{\delta},
    \end{align} 
    where
  \begin{align}
    \label{Mk}
    \mathcal{M}_t=\sqrt{A\trianglemd{t}+B g\left(\trianglemd{t}\right)+C},
  \end{align}
 $\mathcal{M}$ is defined in (\ref{setting}) and $g\left(\cdot\right)$ is a function defined in (\ref{g function}).
  \end{lemma}
\begin{proof}
  Let $X_k=-A_k\la \rvxi{k}, \nablaf{k}{md}\ra$.
   Note that $\xmd{k}$ and $\vx_{k-1}$ are random variables dependent on $\rvz_1, \cdots, \rvz_{k-1}$ and $\rvxi{k}$ is dependent on $\rvz_1, \cdots, \rvz_k$. 
   It is apparent that $X_k$ is the martingale difference sequence since
    \begin{align*}
      \mathbb{E}_k\left[X_k\right]=-A_k\la \mathbb{E}_k \left[\rvxi{k}\right], \nablaf{k}{md}\ra=0.
    \end{align*}
    Also, by Assumption \ref{assumption 3} and applying Cauchy-Schwarz inequality, we have 
    \begin{align}
         &\mathbb{E}_k\left[\exp\left(\frac{X_k^2}{A_k^2\norm{\nablaf{k}{md}}^2\left(A\trianglemd{k}+B\norm{\nabla \fxmd{k}}^2+C\right)}\right)\right]\notag\\
         \leq&\mathbb{E}_k\left[\exp\left(\frac{A_k^2\norm{\rvxi{k}}^2\norm{\nablaf{k}{md}}^2}{A_k^2\norm{\nablaf{k}{md}}^2\left(A\trianglemd{k}+B\norm{\nabla \fxmd{k}}^2+C\right)}\right)\right]\leq \exp(1)
    \end{align}
    Thus, given any $l\in[T]$, applying Lemma \ref{lemma high probability},  we have that for any $\lambda>0$, \withprobability,
    \begin{align}
        \sumlimits{k=1}{l}X_k\leq &\frac{3\lambda}{4}\sumlimits{k=1}{l} A_k^2\norm{\nablaf{k}{md}}^2\left(A\trianglemd{k}+B\norm{\nabla \fxmd{k}}^2+C\right)+\frac{1}{\lambda}\log\frac{1}{\delta}\notag\\
        \leq & \frac{3\lambda}{4}\sumlimits{k=1}{l} A_k^2 \norm{\nablaf{k}{md}}^2  \mathcal{M}_k^2+\frac{1}{\lambda}\log\frac{1}{\delta},
      \end{align}
      where the second inequality follows from Lemma \ref{lemma A.3}. 
      For any fixed $\lambda$, we can rescale over $\delta$ and have that with probability at least $1-\delta$, for all $l\in[T]$,
\begin{align*}
  \sumlimits{k=1}{l}X_k\leq \frac{3\lambda}{4}\sumlimits{k=1}{l} A_k^2 \norm{\nablaf{k}{md}}^2  \mathcal{M}_k^2+\frac{1}{\lambda}\log\frac{T}{\delta}.
\end{align*}
Let $\lambda=\frac{1}{3A_T\mathcal{M}^2}$, and we obtain the desired result.
  \end{proof}
\begin{lemma} 
  \label{lemma D.2.2}
      Given $T\geq 1$ and $\delta\in(0,1)$, if Assumptions \ref{assumption 1}, \ref{assumption 2} and \ref{assumption 3} hold. Then, \withprobability, for all $l\in[T]$, we have 
    \begin{align}
        \label{6.3}
        \sumlimits{k=1}{l} \left(A_k-A_{k-1}\right)\la \rvxi{k}, \vx^*-\vx_{k-1}\ra \leq \frac{3\log\frac{T}{\delta}}{2\PF} \sumlimits{k=1}{l}A_k\norm{\vx^*-\vx_{k-1}}^2\mathcal{M}_k^2+\frac{\PF}{2},
    \end{align}
    where $\mathcal{M}_k$ is defined in (\ref{Mk}) and $\PF$ is defined in (\ref{PF}).
\end{lemma}
\begin{proof}
    Let $Y_k=\left(A_k-A_{k-1}\right)\la \rvxi{k}, \vx^*-\vx_{k-1}\ra$.
    Note that $\xmd{k}$ and $\vx_{k-1}$ are random variables dependent on $\rvz_1, \cdots, \rvz_{k-1}$ and $\rvxi{k}$ is dependent on $\rvz_1, \cdots, \rvz_k$. 
    It is apparent that $Y_k$ is the martingale difference sequence since
    \begin{align*}
      \mathbb{E}_k\left[Y_k\right]= \left(A_k-A_{k-1}\right)\la \mathbb{E}_k\left[\rvxi{k}\right], \vx^*-\vx_{k-1}\ra=0.
    \end{align*}
    Also, by Assumption \ref{assumption 3} and applying Cauchy-Schwarz inequality, we have 
    \begin{align}
        &\mathbb{E}_k\left[\exp\left(\frac{Y_k^2}{A_k\norm{\vx^*-\vx_{k-1}}^2\left(A\trianglemd{k}+B\norm{\nabla \fxmd{k}}^2+C\right)}\right)\right]\notag\\
        \leq&\mathbb{E}_k\left[\exp\left(\frac{\left(A_k-A_{k-1}\right)^2\norm{\rvxi{k}}^2\norm{\vx^*-\vx_{k-1}}^2}{A_k\norm{\vx^*-\vx_{k-1}}^2\left(A\trianglemd{k}+B\norm{\nabla \fxmd{k}}^2+C\right)}\right)\right]\leq \exp(1),
    \end{align} 
    where the last inequality follows from Lemma \ref{lemma 2.1}.
    Thus, given any $l\in[T]$, applying Lemma \ref{lemma high probability},  we have that for any $\lambda>0$, \withprobability,
    \begin{align}
        \sumlimits{k=1}{l}Y_k\leq & \frac{3\lambda}{4} \sumlimits{k=1}{l}A_k\norm{\vx^*-\vx_{k-1}}^2\left(A\trianglemd{k}+B\norm{\nabla \fxmd{k}}^2+C\right)+\frac{1}{\lambda}\log\frac{1}{\delta}\notag\\
        \leq & \frac{3\lambda}{4} \sumlimits{k=1}{l}A_k\norm{\vx^*-\vx_{k-1}}^2 \mathcal{M}_k^2+\frac{1}{\lambda}\log\frac{1}{\delta},
    \end{align}
    where the second inequality follows from Lemma \ref{lemma A.3} and the definition of $\mathcal{M}_k$ in (\ref{Mk}). 
      For any fixed $\lambda$, we can rescale over $\delta$ and have that with probability at least $1-\delta$, for all $l\in[T]$,
      \begin{align*}
        \sumlimits{k=1}{l}Y_k\leq \frac{3\lambda}{4} \sumlimits{k=1}{l}A_k\norm{\vx^*-\vx_{k-1}}^2 \mathcal{M}_k^2+\frac{1}{\lambda}\log\frac{T}{\delta}.
      \end{align*}
      Let $\lambda=\frac{2\log\frac{T}{\delta}}{\PF}$, and we obtain the desired result.
\end{proof}
 We provide the following lemma for Algorithm \ref{algorithm2}, which is similar to Lemma \ref{lemma B.1} in the deterministic case.
\begin{lemma}
  \label{lemma C.5}
  Let $\left\{\xmd{k}\right\}_{k\in[T]}$ and $\left\{\xag{k}\right\}_{k\in[T]}$ be the two sequences generated by Algorithm \ref{algorithm2}. Then we have that for all $k\in[T]$,
  \begin{align*}
      \norm{\xmd{k}-\xag{k-1}}^2\leq \frac{1}{A_k\cdot A_{k-1}}\sumlimits{i=1}{k-1} \frac{A_i^2\cdot \left(\lambda_i-\beta\right)^2}{A_i-A_{i-1}}\norm{\vg_i}^2.
  \end{align*}
\end{lemma}
\begin{proof}
  Lemma \ref{lemma C.5} can be seen as a corollary of Lemma \ref{lemma B.1}.
  As long as we replace the accurate gradient $\nablaf{k}{md}$ in Lemma \ref{lemma B.1} with the stochastic gradient $\vg_t$, the proof is finished.
\end{proof}
\subsubsection{Convergence analysis}
In the next two lemmas, we assume that $\trianglemd{l}$ is bounded in the first $t$ iterations and derive the iteration sequence based on the above analysis, in preparation for the induction argument in Lemma \ref{lemma C.6}.
\begin{lemma}
  \label{lemma C.3}
    Suppose that $\fxmd{l}-f^*\leq \Ftwo, \forall l\in[t]$. 
    Then, under (\ref{eq 32}), for all $l\in[t]$,
    \begin{align}
      \label{eq 27}
  A_l \triangleag{l}+ \frac{2}{\beta}\norm{\vx_l-\vx^*}^2\leq & A_{l-1}\triangleag{l-1}+ \frac{2}{\beta}\norm{\vx_{l-1}-\vx^*}^2
  -\frac{1}{2}\beta A_l\norm{\nablaf{l}{md}}^2+\frac{1}{2}\beta A_l\norm{\rvxi{l}}^2\notag\\
  &+\la \rvxi{l},-\beta A_l\nablaf{l}{md}+\left(A_l-A_{l-1}\right)\left(\vx^*-\vx_{l-1}\right)\ra.
\end{align}
\end{lemma}
\begin{proof}
    With the assumption that $\trianglemd{l}\leq \Ftwo, \forall l\in[t]$ and applying Corollary \ref{corollary 1}, we have $\norm{\nablaf{l}{md}}\leq \sqrt{g\left(\Ftwo\right)}, \forall l\in[t]$. 
    Therefore,
    \begin{align*}
        \norm{\xag{l}-\xmd{l}}=&\beta\norm{\nablaf{l}{md}+\rvxi{l}}
        \leq  \beta\left(\norm{\nablaf{l}{md}}+\norm{\rvxi{l}}\right)\notag\\
        \leq & \beta\left(\sqrt{g(\Ftwo)}+\mathcal{M}\sqrt{\logte}\right)
        \leq  \min\left\{1/L_1,1/L_2\right\},
    \end{align*}
    where the first inequality follows from the triangle inequality and the second inequality holds since (\ref{eq 32}).
    The last inequality holds since $\beta\leq 1/\mathcal{G}_{1,1}$ and $\beta\leq 1/\mathcal{G}_{1,2}$ with $\mathcal{G}_{1,1}, \mathcal{G}_{1,2}$ defined in (\ref{setting 2}).
    By Lemma \ref{lemma A.2}, we have
    \begin{align}
        \label{6.4}
        \fxag{l} \leq & \fxmd{l}+\la \nablaf{l}{md},\xag{l}-\xmd{l}\ra+\frac{\Ldescentlemma}{2}\norm{\xag{l}-\xmd{l}}^2\notag\\
        = & \fxmd{l}-\beta\norm{\nablaf{l}{md}}^2-\beta\la \nablaf{l}{md},\rvxi{l}\ra \notag\\
        &+\frac{\Ldescentlemma}{2}\beta^2\norm{\nablaf{l}{md}+\rvxi{l}}^2.
    \end{align}
    Note that (\ref{1.8}) is derived from the convexity of $f$ and the iteration step 
    \begin{align*}
      \xmd{t}=\frac{A_{t-1}}{A_t}\xag{t-1}+\left(1-\frac{A_{t-1}}{A_t}\right)\vx_{t-1},
    \end{align*}
    which is the same in Algorithm \ref{algorithm1} and Algorithm \ref{algorithm2}.
    Thus, (\ref{1.8}) holds here.
Combining (\ref{1.8}) and (\ref{6.4}), 
\begin{align}
      \fxag{l} \leq & \frac{A_{l-1}}{A_l} \fxag{l-1}+\left(1-\frac{A_{l-1}}{A_l}\right)f^* + \left(1-\frac{A_{l-1}}{A_l}\right) \la \nablaf{l}{md},\vx_{l-1}-\vx^*\ra \notag\\
       &-\beta\norm{\nablaf{l}{md}}^2-\beta\la \nablaf{l}{md},\rvxi{l}\ra +\frac{\Ldescentlemma}{2}\beta^2\norm{\nablaf{l}{md}+\rvxi{l}}^2.
\end{align}
Also, by the iteration step, we have
\begin{align*}
     &\norm{\vx_{l-1}-\vx^*}^2-2\lambda_l \la \nablaf{l}{md}+\rvxi{l},\vx_{l-1}-x^*\ra+\lambda_l^2 \norm{\nablaf{l}{md}+\rvxi{l}}^2\notag\\
        =&\norm{\vx_{l-1}-\lambda_l \left(\nablaf{l}{md}+\rvxi{l}\right)-\vx^*}^2=\norm{\vx_l-\vx^*}^2.
\end{align*}
Hence, 
 \begin{align}
  \label{eq 30}
    \la \nablaf{l}{md}+\rvxi{l},\vx_{l-1}-\vx^*\ra =\frac{1}{2\lambda_l}\left[\norm{\vx_{l-1}-\vx^*}^2-\norm{\vx_l-\vx^*}^2\right]+\frac{\lambda_l}{2}\norm{\nablaf{l}{md}+\rvxi{l}}^2.
 \end{align}
Combining with the fact that 
\begin{align}
  \label{eq 31}
    \norm{\nablaf{l}{md}+\rvxi{l}}^2
  =  \norm{\nablaf{l}{md}}^2+ 2 \la \rvxi{l}, \nablaf{l}{md} \ra +\norm{\rvxi{l}}^2
   \leq 2\norm{\nablaf{l}{md}}^2+ 2\norm{\rvxi{l}}^2,
\end{align}
%\deleted{we obtain that 
%\begin{align}
%        \fxag{l}\leq & \frac{A_{l-1}}{A_l} \fxag{l-1}+\left(1-\frac{A_{l-1}}{A_l}\right)f^*+\frac{A_l-A_{l-1}}{2A_l\cdot\lambda_l}\left[\norm{\vx_{l-1}-\vx^*}^2-\norm{\vx_l-\vx^*}^2\right]\notag\\
%        & -\beta\left(1-\frac{\left(\Ldescentlemma\right)\beta}{2}-\frac{\lambda_l\left(A_l-A_{l-1}\right)}{2\beta A_l}\right)\norm{\nablaf{l}{md}}^2\notag\\
%        &+\left(\frac{\left(\Ldescentlemma\right)\beta^2}{2}+\frac{\lambda_l\left(A_l-A_{l-1}\right)}{2A_l}\right)\norm{\rvxi{l}}^2\notag\\
%        &+\la \rvxi{l},\left(-\beta+\left(\Ldescentlemma\right)\beta^2+\frac{\lambda_l\left(A_l-A_{l-1}\right)}{A_l}\right)\nablaf{l}{md}\ra\notag\\
%        & +\la \rvxi{l},\frac{A_l-A_{l-1}}{A_l}\left(\vx^*-\vx_{l-1}\right)\ra.
%\end{align}
%}
%\deleted{Using the fact that $\la \va,\vb\ra\leq \norm{\va}\norm{\vb}\leq \frac{1}{2}\left(\norm{\va}^2+\norm{\vb}^2\right)$ derived from Cauchy-Schwarz inequality and Young's inequality,} 
we have
\begin{align}
      \fxag{l}\leq & \frac{A_{l-1}}{A_l} \fxag{l-1}+\left(1-\frac{A_{l-1}}{A_l}\right)f^*+\frac{A_l-A_{l-1}}{2A_l\cdot\lambda_l}\left[\norm{\vx_{l-1}-\vx^*}^2-\norm{\vx_l-\vx^*}^2\right]\notag\\
        & -\beta\left(1-\left(\Ldescentlemma\right)\beta-\frac{\lambda_l\left(A_l-A_{l-1}\right)}{\beta A_l}\right)\norm{\nablaf{l}{md}}^2\notag\\
        &+\left(\left(\Ldescentlemma\right)\beta^2+\frac{\lambda_l\left(A_l-A_{l-1}\right)}{A_l}\right)\norm{\rvxi{l}}^2\notag\\
        & +\la \rvxi{l},-\beta \nablaf{l}{md}+\frac{A_l-A_{l-1}}{A_l}\left(\vx^*-\vx_{l-1}\right)\ra.
\end{align}
Since the setting of $\lambda_l$ in (\ref{lambda}), we have
\begin{align*}
\frac{A_l-A_{l-1}}{2A_l\cdot\lambda_l}= \frac{2}{A_l\cdot \beta},
\end{align*}
and
\begin{align*}
  \frac{A_l-A_{l-1}}{A_l}\lambda_l=\frac{A_l-A_{l-1}}{4A_l}\cdot\beta\left(A_l-A_{l-1}\right)\leq \frac{\beta}{4},
\end{align*}
where the inequality follows from Lemma \ref{lemma 2.1}.
Combining with the constraint that $\beta\leq 1/\mathcal{G}_{1,3}$, we have
\begin{align*}
    \fxag{l}\leq & \frac{A_{l-1}}{A_l} \fxag{l-1}+\left(1-\frac{A_{l-1}}{A_l}\right)f^*+\frac{2}{A_l\cdot \beta}\left[\norm{\vx_{l-1}-\vx^*}^2-\norm{\vx_l-\vx^*}^2\right]\notag\\
        &-\frac{1}{2}\beta\norm{\nablaf{l}{md}}^2+\frac{1}{2}\beta\norm{\rvxi{l}}^2+\la \rvxi{l},-\beta \nablaf{l}{md}+\frac{A_l-A_{l-1}}{A_l}\left(\vx^*-\vx_{l-1}\right)\ra.
\end{align*}
Multiplying $A_l$ on both sides and re-arranging the inequality, we obtain the desired result.
\end{proof}

\begin{lemma}
  \label{lemma C.4}
    Under the condition of \lemmaref{lemma C.3},  let \eqref{eq 32}, \eqref{6.2} and \eqref{6.3}. Then,
    \begin{align}
        \label{eq 6.20}
        \gapdistance{l}{l}+\frac{1}{4}\beta\sumlimits{i=1}{l}A_i\norm{\nablaf{i}{md}}^2\leq \PF, \forall 0\leq l\leq t,
    \end{align}
    where 
   \begin{align}
    \label{PF}
        \PF=&\frac{2\mathcal{C}_2}{\beta} +\frac{17}{2} T A_T\beta\mathcal{M}^2\logte,
\end{align}
and  $\mathcal{C}_2$ is defined in \eqref{setting 4}.
\end{lemma}
\begin{proof}
%    In what follows, we assume \ref{eq 32}, \ref{6.2} and \ref{6.3} always hold and then we deduce \ref{eq 6.20}.
%    Since \ref{eq 32}, \ref{6.2} and \ref{6.3} are probability events and hold with probability at least $1-\delta$ separately, it follows that \ref{eq 6.20} hold with probability at least $1-3\delta$.
    It is apparent that $$\gapdistance{0}{0}\leq \PF.$$
    Suppose that for some $k\in[t-1]$, 
    \begin{align}
      \label{eq 50}
      \gapdistance{l}{l}+\frac{1}{4}\beta\sumlimits{i=1}{l}A_i\norm{\nablaf{i}{md}}^2\leq \PF, \forall 0\leq l\leq k.
    \end{align}
    In what follows, we will bound 
    \begin{align*}
      \gapdistance{k+1}{k+1}+\frac{1}{4}\beta\sumlimits{i=1}{k+1}A_i\norm{\nablaf{i}{md}}^2.
    \end{align*}
    Note that $\fxmd{l}-f^*\leq\Ftwo, \forall l\in[t]$, according to Lemma \ref{lemma C.3},
    (\ref{eq 27}) and $\mathcal{M}_l\leq \mathcal{M}$ hold here for all $l\in[k+1]$. 
    Thus, summing up (\ref{eq 27}) over $l\in[k+1]$, we have
    \begin{align}
        \label{6.21}
        \gapdistance{k+1}{k+1}  \leq \gapdistance{0}{0}-\frac{1}{2}\beta \sumlimits{i=1}{k+1}A_i\norm{\nablaf{i}{md}}^2\notag\\
        +\frac{1}{2}\beta\sumlimits{i=1}{k+1}A_i\norm{\rvxi{i}}^2-\beta\sumlimits{i=1}{k+1}A_i\la \rvxi{i}, \nablaf{i}{md}\ra +\sumlimits{i=1}{k+1}\left(A_i-A_{i-1}\right)\la \rvxi{i},\vx^*-\vx_{i-1}\ra.
    \end{align}
    Applying (\ref{6.2}) and letting $l=k+1$, we have
    \begin{align}
        \label{6.22}
        -\beta\sumlimits{i=1}{k+1}A_i\la  \rvxi{i}, \nablaf{i}{md}\ra \leq & \frac{1}{4A_T\mathcal{M}^2}\beta\sumlimits{i=1}{k+1}A_i^2\norm{\nablaf{i}{md}}^2\mathcal{M}_i^2+3A_T\beta\mathcal{M}^2\log \frac{T}{\delta}\notag\\
        \leq & \frac{1}{4}\beta \sumlimits{i=1}{k+1} A_i \norm{\nablaf{i}{md}}^2+3A_T\beta\mathcal{M}^2\log \frac{T}{\delta},
    \end{align}
    where the second inequality follows from $\mathcal{M}_i\leq \mathcal{M}$ and $A_i\leq A_T$ for all $i\in[k+1]$.
    Similarly, applying (\ref{6.3}), we obtain that
    \begin{align}
        \label{6.23}
        \sumlimits{i=1}{k+1}\left(A_i-A_{i-1}\right)\la \rvxi{i},\vx^*-\vx_{i-1}\ra \leq & \frac{3\log \frac{T}{\delta}}{2\PF} \sumlimits{i=1}{k+1}A_i\norm{\vx^*-\vx_{i-1}}^2\mathcal{M}_i^2+\frac{\PF}{2}\notag\\
        \leq &  \frac{3}{4}\beta\log \frac{T}{\delta} \sumlimits{i=1}{k+1}A_i \mathcal{M}_i^2+\frac{\PF}{2}\notag\\
        \leq &  \frac{3}{4}\beta T \cdot A_T  \mathcal{M}^2\log \frac{T}{\delta}+\frac{\PF}{2},
    \end{align}
    where the second inequality holds since
    \begin{align*}
      \norm{\vx_i-\vx^*}^2\leq \frac{1}{2}\beta\cdot\PF, \quad \forall 0\leq i\leq k,
    \end{align*}
    derived from \eqref{eq 50}, and the last inequality follows from $\mathcal{M}_i\leq \mathcal{M}$ and $A_i\leq A_T$ for all $i\in[k+1]$.
    Combining \eqref{6.21}, \eqref{6.22} and \eqref{6.23}, we have
    \begin{align*}
       & \gapdistance{k+1}{k+1} \notag\\
        \leq &  \gapdistance{0}{0}
        -\frac{1}{2}\beta \sumlimits{i=1}{k+1}A_i\norm{\nablaf{i}{md}}^2+\frac{1}{2}\beta\sumlimits{i=1}{k+1}A_i\norm{\rvxi{i}}^2\notag\\
        &+ \frac{1}{4}\beta\sumlimits{i=1}{k+1}A_i\norm{\nablaf{i}{md}}^2+3A_T\beta\mathcal{M}^2\log \frac{T}{\delta} + \frac{3}{4}\beta T \cdot A_T  \mathcal{M}^2\log \frac{T}{\delta}+\frac{\PF}{2}.
    \end{align*}
 Applying \eqref{6.2} with the assumption that $\trianglemd{l}\leq \Ftwo, \forall l\in[t]$, 
 \begin{align*}
  \norm{\rvxi{l}}^2\leq \mathcal{M}^2\logte.
 \end{align*}
 Combining the above inequalities, we obtain that
    \begin{align*}
       &\gapdistance{k+1}{k+1}\notag\\
       \leq & \gapdistance{0}{0}-\frac{1}{4}\beta\sumlimits{i=1}{k+1}A_i\norm{\nablaf{i}{md}}^2\notag\\
        &+\left(\frac{1}{2}\logte+\frac{3}{4}\log\frac{T}{\delta}\right) T\cdot A_T\beta \mathcal{M}^2+3A_T\beta\mathcal{M}^2\log \frac{T}{\delta}+\frac{\PF}{2}\notag\\
        \leq & \gapdistance{0}{0}-\frac{1}{4}\beta\sumlimits{i=1}{k+1}A_i\norm{\nablaf{i}{md}}^2 + \frac{17}{4} T A_T \beta\mathcal{M}^2\logte+\frac{\PF}{2}.
    \end{align*}
    Hence, we could deduce that 
\begin{align}
    \label{eq 6.29}
    \gapdistance{k+1}{k+1}+\frac{1}{4}\beta\sumlimits{i=1}{k+1}A_i\norm{\nablaf{i}{md}}^2\leq \PF,
\end{align}
since 
\begin{align*}
    \PF=&\frac{2}{\beta} \left(\triangleag{0}+2\norm{\vx_0-\vx^*}^2\right) +\frac{17}{2} T A_T\beta\mathcal{M}^2\logte.
\end{align*}
\end{proof}
Based on previous lemmas, we will provide the upper bound of $\trianglemd{t}$ for all $t\in[T]$.
\begin{lemma}
    \label{lemma C.6}
    Under the condition of Theorem \ref{theorem 2}, let (\ref{eq 32}), (\ref{6.2}) and (\ref{6.3}). Then,
    \begin{align}
      \label{eq C.6}
        \functionvaluegap{t}{md}\leq \Ftwo, \forall t\in[T],
    \end{align}
    where $\Ftwo$ is defined in (\ref{setting 3}).
\end{lemma}
\begin{proof}
%    The following analysis is based on \eqref{eq 6.20}, which holds with probability at least $1-3\delta$.
%    Therefore, \eqref{eq C.6} also holds with probability at least $1-3\delta$.
%    In what follows, we assume \eqref{eq 6.20} always holds.
    It is apparent that $\fxmd{1}-f^*=\fxag{0}-f^*\leq \Ftwo$.
    Suppose that for some $t\in[T]$, 
    $$\fxmd{l}-f^*\leq \Ftwo, \forall l\in[t].$$
    Then, by \lemmaref{lemma C.4},  \eqref{eq 6.20} holds.
    Next, we will bound $\fxmd{t+1}-f^*$.
    By \lemmaref{lemma C.5}, we have
    \begin{align}
      \label{eq 44}
      \norm{\xmd{t+1}-\xag{t}}^2\leq &\frac{1}{A_{t+1}\cdot A_t}\sumlimits{i=1}{t} \frac{A_i^2\cdot \left(\lambda_i-\beta\right)^2}{A_i-A_{i-1}}\norm{\nablaf{i}{md}+\rvxi{i}}^2\notag\\
    \leq & 2\sumlimits{i=1}{t} \frac{\lambda_i^2+\beta^2}{A_i-A_{i-1}}\norm{\nablaf{i}{md}+\rvxi{i}}^2,
    \end{align}
    where the second inequality holds since $\left(a-b\right)^2\leq 2a^2+2b^2$ and $A_i\leq A_t\leq A_{t+1}, \forall i\in[t]$.
    Also, since $\lambda_t=\frac{1}{4}\beta\left(A_t-A_{t-1}\right)$ for all $t\in[T]$, we have
    \begin{align*}
      \norm{\xmd{t+1}-\xag{t}}^2\leq & 2\beta^2\sumlimits{i=1}{t}\frac{\frac{1}{16}\left(A_i-A_{i-1}\right)^2+1}{A_i-A_{i-1}}\norm{\nablaf{i}{md}+\rvxi{i}}^2\notag\\
      \leq & \frac{17}{8}\beta^2\sumlimits{i=1}{t}\left(A_i-A_{i-1}\right)\norm{\nablaf{i}{md}+\rvxi{i}}^2\notag\\
      \leq & \frac{17}{4}\beta^2\sumlimits{i=1}{t}\left(A_i-A_{i-1}\right)\left(\norm{\nablaf{i}{md}}^2+\norm{\rvxi{i}}^2\right),
    \end{align*}
    where the second inequality follows from \lemmaref{lemma 2.1} and the last inequality holds since $\norm{\va+\vb}^2\leq 2\norm{\va}^2+2\norm{\vb}^2$.
    Applying \lemmaref{lemma 2.1} and using the fact that $\sqrt{\beta A_i}=\sqrt{\beta B_i+1}\geq1, \forall i\in[T]$, we have
    \begin{align}
      \label{eq 58}
      \norm{\xmd{t+1}-\xag{t}}^2\leq & \frac{17}{4}\beta^2 \sumlimits{i=1}{t}\sqrt{A_i}\left(\norm{\nablaf{i}{md}}^2+\norm{\rvxi{i}}^2\right)
      \notag\\
      \leq & \frac{17}{4}\beta^{\frac{5}{2}}\sumlimits{i=1}{t}A_i\left(\norm{\nablaf{i}{md}}^2+\norm{\rvxi{i}}^2\right).
    \end{align}
    Since the assumption that $\fxmd{l}-f^*\leq \Ftwo, \forall l\in[t]$, by \eqref{eq 6.20}, we have 
    \begin{align*}
      \beta \sumlimits{i=1}{t}A_i\norm{\nablaf{i}{md}}^2\leq 4\PF.
    \end{align*}
    Combining with \eqref{eq 58}, \eqref{eq 32} and recalling the expression of $\PF$ in \eqref{PF}, we obtain that 
    \begin{align*}
      \norm{\xmd{t+1}-\xag{t}}^2\leq & 17\beta^\frac{3}{2}\PF+ \frac{17}{4} \beta^{\frac{5}{2}}  T  A_T \mathcal{M}^2\logte\notag\\
      = & 34\sqrt{\beta} \cdot \mathcal{C}_2 + \frac{289}{2}\beta^\frac{5}{2} T  A_T  \mathcal{M}^2\logte + \frac{17}{4}\beta^\frac{5}{2} T  A_T \mathcal{M}^2\logte\notag\\
      = & 34\sqrt{\beta}\cdot \mathcal{C}_2 + \frac{595}{4}\beta^\frac{5}{2} T  A_T  \mathcal{M}^2\logte.
    \end{align*}
  Combining with \lemmaref{lemma 2.1} and the setting that $A_T=B_T+1/\beta$, we have
  \begin{align*}
    \norm{\xmd{t+1}-\xag{t}}^2  \leq  34\sqrt{\beta}\cdot  \mathcal{C}_2 + \frac{595}{4}\beta^\frac{5}{2} T^3 \mathcal{M}^2\logte +\frac{595}{4}\beta^{\frac{3}{2}}T\mathcal{M}^2\logte.
  \end{align*}
    Since $\beta \leq \min\left\{1/\mathcal{G}_{1,4},1/\left(\mathcal{G}_2T^{\frac{6}{5}}\right),1 / \left(\mathcal{G}_3T^{\frac{2}{3}}\right)\right\}$, where $\mathcal{G}_{1,4}, \mathcal{G}_2, \mathcal{G}_3$ are defined in \eqref{setting}, \eqref{setting 2}, 
    \begin{align*}
      \norm{\xmd{t+1}-\xag{t}}^2\leq \frac{1}{\left(L_1+L_2\right)^2}.
    \end{align*}
    Hence, applying \lemmaref{lemma A.2} and Cauchy-Schwarz inequality, we have
    \begin{align}
      \label{eq 45}
      \fxmd{t+1} \leq & \fxag{t}+\la \nablaf{t}{ag}, \xmd{t+1}-\xag{t}\ra + \frac{L_0+L_1\norm{\nablaf{t}{ag}}^p+L_2\left(\triangleag{t}\right)^q}{2}\norm{\xmd{t+1}-\xag{t}}^2\notag\\
      \leq & \fxag{t}+ \norm{\nablaf{t}{ag}}\cdot \norm{\xmd{t+1}-\xag{t}} +\frac{L_0+L_1\norm{\nablaf{t}{ag}}^p+L_2\left(\triangleag{t}\right)^q}{2}\norm{\xmd{t+1}-\xag{t}}^2\notag\\
      \leq & \fxag{t}+\frac{1}{L_1+L_2}\norm{\nablaf{t}{ag}}+\frac{L_0+L_1\norm{\nablaf{t}{ag}}^p+L_2\left(\triangleag{t}\right)^q}{2\left(L_1+L_2\right)^2}.
    \end{align}
    Since the assumption that $\trianglemd{l}\leq \Ftwo, \forall l\in[t]$, by \lemmaref{lemma C.4}, we have
    \begin{align}
      \label{eq 60}
      \triangleag{t}\leq \frac{\PF}{A_t} \leq & \beta\cdot \PF,
    \end{align}
    where the second inequality holds since $A_t\geq 1/\beta$.
    Plugging \eqref{PF} into \eqref{eq 60}, we obtain that
    \begin{align*}
      \triangleag{t}\leq & 2\mathcal{C}_2 + \frac{17}{2}T^3\beta^2 \mathcal{M}^2\logte+\frac{17}{2}T\beta\mathcal{M}^2\logte 
      \leq  2\mathcal{C}_2+ 17 \logte=\mathcal{H},
    \end{align*}
    where the last inequality follows from 
    $$\beta\leq \frac{1}{\mathcal{M}T^{\frac{3}{2}}}, \quad \text{and} \quad \beta\leq \frac{1}{\mathcal{M}^2T}.$$ 
    Note that $\mathcal{H}$ is independent on $\Ftwo$.
    By \corollaryref{corollary 1}, we have
   $
      \norm{\nablaf{t}{ag}}\leq \sqrt{g(\mathcal{H})}.$
    Combining with \eqref{eq 45} and subtracting $f^*$ from both sides, we obtain that
    \begin{align*}
      \trianglemd{t+1}\leq \mathcal{H}+\frac{\sqrt{g(\mathcal{H})}}{L_1+L_2}+\frac{L_0+L_1\left(g(\mathcal{H})\right)^{\frac{p}{2}}+L_2\mathcal{H}^q}{2\left(L_1+L_2\right)^2} = \Ftwo.
    \end{align*}
    Now we finish the induction and obtain the desired result.
\end{proof}
With the above lemmas, we are ready to prove the final convergence result.
\begin{proof}[Proof of Theorem \ref{theorem 2}]
    In what follows, we assume (\ref{eq 32}), (\ref{6.2}) and (\ref{6.3}) always hold, and under these conditions we prove the desired error bounds.
    Using Lemmas \ref{lemma C.1}, \ref{lemma 6.1} and \ref{lemma D.2.2}, (\ref{eq 32}), (\ref{6.2}) and (\ref{6.3}) hold with probability at least $1-3\delta$.
    Thus, the desired error bounds also hold with probability at least  $1-3\delta$.
    By Lemma \ref{lemma C.6},
    (\ref{eq C.6}) holds. 
    Based on Lemma \ref{lemma C.4}, we obtain that
    \begin{align*}
      \triangleag{T}\leq \frac{\PF}{A_T}\leq \frac{8\mathcal{C}_2}{T^2\beta} + \frac{17}{2}T\beta\mathcal{M}^2\logte.
    \end{align*}
    Since the constraints of $\beta$ in (\ref{beta}), we have
    \begin{align}
      \label{Complete theorem 2}
      \triangleag{T} \leq & \frac{8\mathcal{C}_2}{T^2}\left(L_1+L_2\right)\left(\sqrt{g\left(\Ftwo\right)}+\mathcal{M}\sqrt{\logte}\right) \notag\\
      & + \frac{32\mathcal{C}_2}{T^2} \left(L_0+L_1\left(g\left(\Ftwo\right)\right)^{\frac{p}{2}}+L_2\Ftwo^q +    1156\left(L_1+L_2\right)^4\mathcal{C}_2^2 \right)\notag\\
      & + \frac{8\mathcal{C}_2}{T^\frac{4}{5}}  \left(595\right)^{\frac{2}{5}}\left(L_1+L_2\right)^{\frac{4}{5}}\mathcal{M}^{\frac{4}{5}}\left(\logte\right)^{\frac{2}{5}}  + \frac{8\mathcal{C}_2}{T^{\frac{4}{3}}} \left(595\right)^{\frac{2}{3}}\left(L_1+L_2\right)^{\frac{4}{3}}\mathcal{M}^{\frac{4}{3}}\left(\logte\right)^{\frac{2}{3}}   \notag\\
      & +  \frac{8\mathcal{C}_2\mathcal{M}^2}{T}+ \frac{\mathcal{M}}{\sqrt{T}}  \left(\frac{17}{2}\logte  +  8\mathcal{C}_2\right).
    \end{align}
\end{proof}

\section{Convergence of SNAG under Lipschitz smoothness for stochastic optimization}
\label{section 5}
\begin{sloppypar}
  We apply our analysis to smooth stochastic optimization and demonstrate that SNAG could reach the accelerated convergence rate in expectation under the relaxed affine-variance noises and standard smoothness.
\end{sloppypar}
\subsection{Convergence results}
\begin{theorem}
  \label{theorem 3}
Let $T>0$, $f$ be $L$-smooth and convex. Suppose that $\left\{\xag{t}\right\}_{t\in[T]}$ is a sequence generated by \Algref{algorithm2} with step sizes 
\begin{align}
  \label{setting2}
  \beta=\min\left\{\frac{1}{2L\left(1+B\right)},\frac{1}{\mathcal{Q}^{\frac{1}{2}}T^{\frac{3}{2}}},\frac{1}{\mathcal{Q}T}\right\},
\end{align}
\begin{align}
  \label{eq 61}
  \lambda_t=\frac{\beta}{2\left(1+B\right)}\left(A_t-A_{t-1}\right),
\end{align}
where $\mathcal{Q}=A\Fthree+C$ is a constant depending on the parameters of smoothness and noise with $\Fthree$ defined in (\ref{Fthree}).
Under Assumptions \ref{assumption 1}, \ref{assumption 2} and \ref{assumption 4}, we have\footnote{The detail convergence result is presented in (\ref{Complete theorem 3}).}
\iffalse 
\begin{align*}
      \mathbb{E}\left[\fxag{T}-f^*\right]\leq & \mathcal{O}\left(\frac{L}{T^2}+\sqrt{\frac{A+B+C}{T}}\right).
  \end{align*}
\fi
\begin{align}
  \label{eq 11}
      \mathbb{E}\left[\fxag{T}-f^*\right]\leq & \mathcal{O}\left(\frac{1+B}{T^2}+\sqrt{\frac{A+C}{T}}\right).
  \end{align}
\end{theorem}
The above theorem relaxes the bounded variance noise assumption in \cite{ghadimi2016accelerated} while providing the optimal expected convergence rate.
Furthermore, Theorem \ref{theorem 3} improves the convergence rate of order $\mathcal{O}\left(1/T+C/\sqrt{T}\right)$ for SGD and SNAG in \cite{yu2025convergence} under the same assumption.
Compared to Theorem \ref{theorem 2}, the suboptimal term $\mathcal{O}\left(\sqrt{B/T}\right)$ with respect to $B$ disappears in (\ref{eq 11}), which aligns with the expected result of $\mathcal{O}\left(\left(B+1\right)/T^2\right)$ and almost-sure result of $o\left(\left(B+1\right)/T^2\right)$ in \cite{hermant2025gradient} where they focused on smooth stochastic optimization with noise satisfying (\ref{strong growth condition}).
\subsection{Proof of Theorem \ref{theorem 3}}
\label{proof of theorem 3}
We first provide the following lemma as a key to the induction argument in Lemma \ref{lemma D.2}.
\begin{lemma}
  \label{lemma D.1}
  Under the conditions of Theorem \ref{theorem 3}, for all $t\in[T]$, it holds that
 \begin{align*}
  &\mathbb{E}\left[A_t \triangleag{t}\right]+ \frac{1+B}{\beta}\mathbb{E}\left[\norm{\vx_t-\vx^*}^2\right] \notag\\
  \leq& \frac{\mathcal{C}_3}{\beta}-\frac{1}{2}\beta \sumlimits{l=1}{t}A_l\mathbb{E}\left[\norm{\nablaf{l}{md}}^2\right]+\frac{1}{2\left(1+B\right)}\beta \sumlimits{l=1}{t}A_l \mathbb{E}\left[A\trianglemd{l}+C\right],
\end{align*}
where 
 \begin{align}
    \label{C3}
    \mathcal{C}_3=\triangleag{0}+\left(1+B\right)\norm{\vx_0-\vx^*}^2.
  \end{align}
\end{lemma}
\begin{proof}
  By the descent lemma for Lipschitz smooth functions and the iteration step in Algorithm \ref{algorithm2}, we have
  \begin{align*}
    \fxag{l}\leq & \fxmd{l}+\la \nablaf{l}{md},\xag{l}-\xmd{l}\ra+\frac{L}{2}\norm{\xag{l}-\xmd{l}}^2\notag\\
        = & \fxmd{l}-\beta\norm{\nablaf{l}{md}}^2-\beta\la \nablaf{l}{md},\rvxi{l}\ra+\frac{L}{2}\beta^2\norm{\nablaf{l}{md}+\rvxi{l}}^2.
  \end{align*}
  Note that (\ref{1.8}), (\ref{eq 30}) and (\ref{eq 31}) still holds here as they are independent of the smoothness condition. Thus,
  \begin{align}
    \label{eq 81}
      \fxag{l} \leq & \frac{A_{l-1}}{A_l} \fxag{l-1}+\left(1-\frac{A_{l-1}}{A_l}\right)f^* + \left(1-\frac{A_{l-1}}{A_l}\right) \la \nablaf{l}{md},\vx_{l-1}-\vx^*\ra \notag\\
      & -\beta\norm{\nablaf{l}{md}}^2-\beta\la \nablaf{l}{md},\rvxi{l}\ra +\frac{L}{2}\beta^2\norm{\nablaf{l}{md}+\rvxi{l}}^2\notag\\
      \leq & \frac{A_{l-1}}{A_l} \fxag{l-1}+\left(1-\frac{A_{l-1}}{A_l}\right)f^*+\frac{A_l-A_{l-1}}{2A_l\cdot\lambda_l}\left[\norm{\vx_{l-1}-\vx^*}^2-\norm{\vx_l-\vx^*}^2\right]\notag\\
        & -\beta\left(1-\frac{L\beta}{2}-\frac{\lambda_l\left(A_l-A_{l-1}\right)}{2\beta A_l}\right)\norm{\nablaf{l}{md}}^2+\left(\frac{L\beta^2}{2}+\frac{\lambda_l\left(A_l-A_{l-1}\right)}{2A_l}\right)\norm{\rvxi{l}}^2\notag\\
        &+\la \rvxi{l},\left(-\beta+L\beta^2+\frac{\lambda_l\left(A_l-A_{l-1}\right)}{A_l}\right)\nablaf{l}{md}\ra +\la \rvxi{l},\frac{A_l-A_{l-1}}{A_l}\left(\vx^*-\vx_{l-1}\right)\ra.
  \end{align}
By Assumption \ref{assumption 4}, we obtain that for all $l\in[T]$, 
\begin{align}
  \label{eq 82.1}
  \mathbb{E}\left[\norm{\rvxi{_l}}^2\right]=\mathbb{E}\left[\mathbb{E}_l\left[\norm{\rvxi{l}}^2\right] \right]\leq \mathbb{E}\left[A \trianglemd{l}+B \norm{\nablaf{l}{md}}^2+C\right].
\end{align}
With multiplying $A_l$ and taking expectation on both sides of (\ref{eq 81}), we have 
\begin{align}
  \label{eq 82}
   \mathbb{E}\left[A_l \triangleag{l}\right] \leq & \mathbb{E} \left[A_{l-1} \triangleag{l-1}\right] +\frac{A_l-A_{l-1}}{2\lambda_l}\mathbb{E}\left[\norm{\vx_{l-1}-\vx^*}^2-\norm{\vx_l-\vx^*}^2\right] \notag\\
  & - \beta A_l\left(1-\frac{L\beta}{2}-\frac{\lambda_l\left(A_l-A_{l-1}\right)}{2\beta A_l}\right)\mathbb{E} \left[\norm{\nablaf{l}{md}}^2\right] \notag\\
  & +\beta A_l \left(\frac{L\beta}{2}+\frac{\lambda\left(A_l-A_{l-1}\right)}{2\beta A_l}\right)\mathbb{E}\left[\norm{\rvxi{l}}^2\right]\notag\\
  \leq & \mathbb{E} \left[A_{l-1} \triangleag{l-1}\right] +\frac{A_l-A_{l-1}}{2\lambda_l}\mathbb{E}\left[\norm{\vx_{l-1}-\vx^*}^2-\norm{\vx_l-\vx^*}^2\right] \notag\\
  & - \beta A_l\left(1-\frac{L\beta}{2}-\frac{\lambda_l\left(A_l-A_{l-1}\right)}{2\beta A_l}\right)\mathbb{E} \left[\norm{\nablaf{l}{md}}^2\right]\notag\\
  & + \beta A_l \left(\frac{L\beta}{2}+\frac{\lambda_l\left(A_l-A_{l-1}\right)}{2\beta A_l}\right)  \mathbb{E}\left[A \trianglemd{l}+B \norm{\nablaf{l}{md}}^2+C\right]\notag\\
  = & \mathbb{E} \left[A_{l-1} \triangleag{l-1}\right] +\frac{A_l-A_{l-1}}{2\lambda_l}\mathbb{E}\left[\norm{\vx_{l-1}-\vx^*}^2-\norm{\vx_l-\vx^*}^2\right] \notag\\
  & - \beta A_l \left(1- \left(1+B\right) \left(\frac{L\beta}{2}+\frac{\lambda_l\left(A_l-A_{l-1}\right)}{2\beta A_l}\right) \right) \mathbb{E}\left[\norm{\nablaf{l}{md}}^2\right] \notag\\
  & + \beta A_l \left(\frac{L\beta}{2}+\frac{\lambda_l\left(A_l-A_{l-1}\right)}{2\beta A_l}\right)  \mathbb{E} \left[A \trianglemd{l}+C\right],
\end{align} 
where the second inequality follows from (\ref{eq 82.1}).
Since $\lambda_k=\frac{\beta}{2(1+B)}\left(A_k-A_{k-1}\right)$, we have 
\begin{align*}
\frac{A_l-A_{l-1}}{2\lambda_l}= \frac{1+B}{\beta},
\end{align*}
and
\begin{align*}
  \frac{\lambda_l\left(A_l-A_{l-1}\right)}{2\beta A_l}=\frac{\left(A_l-A_{l-1}\right)^2}{4 A_l\left(1+B\right)}\leq \frac{1}{4\left(1+B\right)},
\end{align*}
where the inequality follows from Lemma \ref{lemma 2.1}.
Combining with (\ref{eq 82}), we have 
\begin{align*}
\mathbb{E}\left[A_l \triangleag{l}\right]  \leq &  \mathbb{E} \left[A_{l-1} \triangleag{l-1}\right] +\frac{1+B}{\beta}\mathbb{E}\left[\norm{\vx_{l-1}-\vx^*}^2-\norm{\vx_l-\vx^*}^2\right] \notag\\
& - \beta A_l\left(1-\left(1+B\right)\left(\frac{L\beta}{2}+\frac{1}{4\left(1+B\right)}\right)\right) \mathbb{E}\left[\norm{\nablaf{l}{md}}^2\right]\notag\\
& + \beta A_l\left(\frac{L \beta}{2}+\frac{1}{4\left(1+B\right)}\right) \mathbb{E} \left[A \trianglemd{l}+C\right]\notag\\
\leq & \mathbb{E} \left[A_{l-1} \triangleag{l-1}\right] +\frac{1+B}{\beta}\mathbb{E}\left[\norm{\vx_{l-1}-\vx^*}^2-\norm{\vx_l-\vx^*}^2\right] \notag\\
& - \frac{1}{2}\beta A_l \mathbb{E}\left[\norm{\nablaf{l}{md}}^2\right]+\frac{1}{2\left(1+B\right)}\beta A_l \mathbb{E}\left[A\trianglemd{l}+C\right],
\end{align*}
where the last inequality follows from $\beta \leq \frac{1}{2L\left(1+B\right)}$.
Re-arranging the above inequality and summing up over $l\in[t]$, we obtain that 
\begin{align*}
  &\mathbb{E}\left[A_t \triangleag{t}\right]+ \frac{1+B}{\beta}\mathbb{E}\left[\norm{\vx_t-\vx^*}^2\right]\notag\\
  \leq & A_0\triangleag{0}+\frac{1+B}{\beta}\norm{\vx_0-\vx^*}^2-\frac{1}{2}\beta \sumlimits{l=1}{t}A_l\mathbb{E}\left[\norm{\nablaf{l}{md}}^2\right] +\frac{1}{2\left(1+B\right)}\beta \sumlimits{l=1}{t}A_l \mathbb{E}\left[A\trianglemd{l}+C\right]\notag\\
  = & \frac{\mathcal{C}_3}{\beta}-\frac{1}{2}\beta \sumlimits{l=1}{t}A_l\mathbb{E}\left[\norm{\nablaf{l}{md}}^2\right]+\frac{1}{2\left(1+B\right)}\beta \sumlimits{l=1}{t}A_l \mathbb{E}\left[A\trianglemd{l}+C\right],
\end{align*}
where the last line holds since $A_0=1/\beta$.
\end{proof}
Similar to Lemma \ref{lemma C.6}, we will bound the function value gap in expectation by induction.
\begin{lemma}
  \label{lemma D.2}
  Under the condition of Theorem \ref{theorem 3}, we have 
  \begin{align*}
    \mathbb{E}\left[\fxmd{t}-f^*\right]\leq \Fthree, \forall t\in[T],
  \end{align*}
  where 
  \begin{align}
    \label{Fthree}
  \Fthree=\left(2+5\sqrt{2L\left(1+B\right)}\right)\mathcal{C}_3+1+10\sqrt{2L},
  \end{align}
  with $\mathcal{C}_3$ defined in (\ref{C3}).
\end{lemma}
\begin{proof}
  We will prove this lemma by induction.
    Obviously, we have
    \begin{align*}
      \mathbb{E}\left[\fxmd{1}-f^*\right]=\fxag{0}-f^*\leq \Fthree.
    \end{align*}
    Suppose that for some $t\in[T]$, 
    \begin{align*}
       \mathbb{E}\left[\fxmd{l}-f^*\right]\leq \Fthree, \forall l\in[t].
    \end{align*}
    Next, we will bound $\mathbb{E}\left[\fxmd{t+1}-f^*\right]$.
    Since (\ref{eq 44}) is independent of the smoothness condition, it still holds here.
    \begin{align*}
      \norm{\xmd{t+1}-\xag{t}}^2 \leq & 2\sumlimits{i=1}{t} \frac{\lambda_i^2+\beta^2}{A_i-A_{i-1}}\norm{\nablaf{i}{md}+\rvxi{i}}^2\notag\\
      \leq & 2\sumlimits{i=1}{t}\frac{\frac{1}{4}\left(A_i-A_{i-1}\right)^2\beta^2+\beta^2}{A_i-A_{i-1}}\norm{\nablaf{i}{md}+\rvxi{i}}^2\notag\\
      \leq & \frac{5}{2}\beta^2\sumlimits{i=1}{t} \left(A_i-A_{i-1}\right)\norm{\nablaf{i}{md}+\rvxi{i}}^2\notag\\
      \leq & 5\beta^2\sumlimits{i=1}{t}\left(A_i-A_{i-1}\right)\left(\norm{\nablaf{i}{md}}^2+\norm{\rvxi{i}}^2\right),
    \end{align*}
    where the second inequality holds since the constraint of $\lambda_i$ in \eqref{eq 61}, the third inequality follows from Lemma \ref{lemma 2.1} and the last inequality holds since $\norm{\va+\vb}^2\leq 2\left(\norm{\va}^2+\norm{\vb}^2\right)$.
    Applying Lemma \ref{lemma 2.1} again and using the fact that
    \begin{align*}
      \sqrt{\beta A_t}=\sqrt{\beta \left(B_t+1/\beta\right)}\geq 1, \forall t\in[T],
    \end{align*}
    we have 
    \begin{align}
      \label{eq 47}
      \norm{\xmd{t+1}-\xag{t}}^2\leq & 5\beta^2\sumlimits{i=1}{t}\sqrt{A_i}\left(\norm{\nablaf{i}{md}}^2+\norm{\rvxi{i}}^2\right)\notag\\
      \leq  & 5\beta^{\frac{5}{2}}\sumlimits{i=1}{t}A_i\left(\norm{\nablaf{i}{md}}^2+\norm{\rvxi{i}}^2\right).
    \end{align}
    Taking expectation on both sides of the above inequality and combining with (\ref{eq 82.1}), we obtain that
    \begin{align*}
       \mathbb{E}\left[\norm{\xmd{t+1}-\xag{t}}^2\right] \leq & 5 \beta^{\frac{5}{2}}\sumlimits{i=1}{t}A_i \left(\mathbb{E}\left[\norm{\nablaf{i}{md}}^2\right]+\mathbb{E}\left[A\trianglemd{i}+B\norm{\nablaf{i}{md}}^2+C\right]\right)\notag\\
      = & 5 \beta^{\frac{5}{2}}\left(1+B\right)\sumlimits{i=1}{t}A_i\mathbb{E} \left[\norm{\nabPsi{i}{md}}^2\right]+5\beta^{\frac{5}{2}}\sumlimits{i=1}{t}A_i\mathbb{E}\left[A\trianglemd{i}+C\right]\notag\\
      \leq & 10\beta^{\frac{1}{2}}\left(1+B\right)\mathcal{C}_3+10\beta^{\frac{5}{2}}\sumlimits{i=1}{t}A_i\mathbb{E}\left[A\trianglemd{i}+C\right]\notag\\
      \leq & 10\beta^{\frac{1}{2}}\left(1+B\right)\mathcal{C}_3+10\beta^{\frac{5}{2}}\cdot T \cdot A_T \left(A\Fthree+C\right)\notag\\
      \leq & 10\beta^{\frac{1}{2}}\left(1+B\right)\mathcal{C}_3+10\beta^{\frac{5}{2}}T^3\left(A\Fthree+C\right)+10\beta^{\frac{3}{2}}T\left(A\Fthree+C\right),
    \end{align*}
    where the second inequality follows from Lemma \ref{lemma D.1}, the third inequality holds since $A_i\leq A_T, \forall i\in[t]$ and the assumption that $\mathbb{E}\left[\trianglemd{i}\right]\leq \Fthree, \forall i\in[t]$, and the last inequality follows from Lemma \ref{lemma 2.1} with $A_t=B_t+1/\beta, \forall t\in[T]$.
    Since the constraints of $\beta$ in (\ref{setting2}), we have
    \begin{align*}
      \mathbb{E}\left[\norm{\xmd{t+1}-\xag{t}}^2\right]\leq & \frac{5\sqrt{2\left(1+B\right)}}{\sqrt{L}}\mathcal{C}_3+\frac{10\sqrt{2}}{\sqrt{L\left(1+B\right)}}.
    \end{align*}
    Applying the descent lemma again, we obtain that 
    \begin{align*}
      \fxmd{t+1}\leq & \fxag{t}+\la \nablaf{t}{ag}, \xmd{t+1}-\xag{t}\ra+\frac{L}{2}\norm{\xmd{t+1}-\xag{t}}^2\notag\\
      \leq & \fxag{t}+\norm{\nablaf{t}{ag}}\cdot \norm{\xmd{t+1}-\xag{t}}+\frac{L}{2}\norm{\xmd{t+1}-\xag{t}}^2\notag\\
      \leq & \fxag{t}+\frac{1}{2L}\norm{\nablaf{t}{ag}}^2+\frac{L}{2}\norm{\xmd{t+1}-\xag{t}}^2+\frac{L}{2}\norm{\xmd{t+1}-\xag{t}}^2\notag\\
      \leq & \fxag{t}+\left(\fxag{t}-f^*\right)+L\cdot \norm{\xmd{t+1}-\xag{t}}^2,
    \end{align*}
    where we apply Cauchy-Schwarz inequality in the second inequality and apply Young's inequality in the third line. The last inequality follows from Lemma \ref{Lemma A.1}.
    Subtracting $f^*$ from both sides and taking expectation, we have 
    \begin{align*}
      \mathbb{E}\left[\trianglemd{t+1}\right]\leq & 2\mathbb{E}\left[\triangleag{t}\right]+L\cdot \mathbb{E}\left[\norm{\xmd{t+1}-\xag{t}}^2\right].
    \end{align*}
    With the assumption that $\fxmd{l}-f^*\leq \Fthree, \forall l\in[t]$ and applying Lemma \ref{lemma D.1}, we obtain that 
    \begin{align*}
      \mathbb{E}\left[\triangleag{t}\right]\leq & \frac{1}{A_t\beta} \mathcal{C}_3+\frac{1}{2A_t\left(1+B\right)}\beta\cdot \sumlimits{i=1}{t}A_i\mathbb{E}\left[A\trianglemd{l}+C\right]\notag\\
      \leq  &\mathcal{C}_3+\frac{1}{2\left(1+B\right)}\beta T\left(A\Fthree+C\right)
      \leq  \mathcal{C}_3+\frac{1}{2},
    \end{align*}
    where the second inequality holds since $A_t\geq 1/\beta$ and $A_i\leq A_t, \forall \in[t]$, and the last line follows from the definition of $\beta$.
    Therefore, we have 
    \begin{align*}
      \mathbb{E}\left[\trianglemd{t+1}\right]\leq & \left(2+5\sqrt{2L\left(1+B\right)}\right)\mathcal{C}_3+1+10\sqrt{2L}=\Fthree.
    \end{align*}
    Now we finish the induction and obtain the desired result.
\end{proof}
Based on Lemma \ref{lemma D.1} and Lemma \ref{lemma D.2}, we could obtain the final convergence rate.
\begin{proof}[Proof of Theorem  \ref{theorem 3}]
    By Lemma \ref{lemma D.2}, we have $\mathbb{E}\left[\fxmd{t}-f^*\right]\leq \Fthree, \forall t\in[T]$.
    Then, combining Lemma \ref{lemma D.1}, Assumption \ref{assumption 4} and the fact that $A_t\leq A_T, \forall t\in[T]$, we obtain that  
    \begin{align}
      \label{Complete theorem 3}
      \mathbb{E}\left[\fxag{T}-f^*\right]\leq & \frac{1}{A_T\beta}\mathcal{C}_3+\frac{\beta}{2A_T\left(1+B\right)}T\cdot A_T \left(A\Fthree+C\right)\notag\\
      \leq & \frac{8L\left(1+B\right)\mathcal{C}_3}{T^2}+\frac{4\mathcal{C}_3\mathcal{Q}}{T}+\frac{4\mathcal{C}_3\sqrt{\mathcal{Q}}}{\sqrt{T}}+ \frac{\sqrt{\mathcal{Q}}}{2\sqrt{T}},
    \end{align}
    where the second inequality holds since Lemma \ref{lemma 2.1} and the setting of $\beta$ in (\ref{setting2}).
\end{proof}

%We refer to \Secref{proof of theorem 3} for the complete proof of \ref{theorem 3}.

\section{Numerical experiment}
\label{numerical experiment}
In this section, we show the practical convergence behavior of SNAG with two different step setting (see Algorithm \ref{algorithm2}, Algorithm \ref{algorithm3}) compared to SGD, i.e., 
\begin{align}
  \vx_{t+1}=\vx_t-\eta \nabla_\vz f(\vx_t;\vz_t),
\end{align}
on the two-layer neural network (\ref{two layer}) and phase retrieval model (\ref{phase retrieval}).
We prove that both the two models satisfy the $\left(L_0,L_1,L_2\right)$-smoothness condition in Section \ref{section B}.

\textbf{Two-layer neural network}
Considering the following problem,
  \begin{align}
    \label{two layer}
    \min\limits_{\vx\in\mathbb{R}^{\tilde{d}}}F(\vx)=\frac{1}{n}\sumlimits{i=1}{n}f\left(y_i\Phi\left(\vx,\vw_i\right)\right),
  \end{align}
  where $\vw_i\in\mathbb{R}^d$ is the data point and its associated label $y_i\in\left\{\pm 1\right\}$. The function $f(\cdot)$ is the exponential loss i.e., $f(t)=\exp(-t)$   
  and $\Phi(\cdot)$ is a two-layer neural network with $m$ neurons defined as 
  \begin{align*}
    \Phi(\vx, \vw)=\sumlimits{j=1}{m}a_j\sigma(\la \vx_j,\vw\ra),
  \end{align*}
  Here $\sigma:\mathbb{R}\rightarrow \mathbb{R}$ is the activation function and $\vx_j\in\mathbb{R}^d$ denotes the input weight vector of the $j$th hidden neuron. 
  $\vx\in\mathbb{R}^{\tilde{d}}$ represents the concatenation of these weights, i.e., $\vx=\left[\vx_1;\vx_2;\cdots;\vx_m\right]$ where $\tilde{d}=md$.
  We assume that only $\vx_j$ can be updated during training, while $\va_j\in\mathbb{R}$ are initialized randomly and kept fixed.

We conduct experiment on the specific shallow neural network  with $m=30$ hidden neurons, exponential loss $f(t)=\exp(-t)$ and smoothed-leaky-ReLU activation function, i.e.,
\begin{align*}
  \sigma(t)= t \mathbb{I}(t\geq 0)+0.2 t\mathbb{I}(t<0),
\end{align*}
where $\mathbb{I}(\cdot)$ is the $0-1$ indicator function. 
We generate the data point $\vw_i\in\mathbb{R}^d$, where dimension $d=10$, coordinate-wise from Gaussian distribution $\mathcal{N}\left(0,25\right)$ with its binary label $y_i\in\left\{\pm 1\right\}$ chosen randomly. 
The second layer weights are generated randomly from $a_j\in\left\{\pm \frac{1}{m}\right\}$ and kept fixed during training.

\textbf{Phase retrieval}
Phase retrieval is a classic model in the field of machine learning and signal processing \cite{drenth1994principles,miao1999extending,chen2023generalized}.
In this setting, we are aimed to solve the following problem, i.e.,
\begin{align}
  \label{phase retrieval}
  \min\limits_{\vx\in\mathbb{R}^d} f(\vx):=\frac{1}{2m}\sumlimits{i=1}{m}\left(y_i-\left|\va_r^\top\vx\right|^2\right)^2.
\end{align}
Here, $y_i$ represents the intensity measurements, i.e., $y_i=\left|\va_i^\top \vz\right|, \forall i\in[m]$ with $\va_i\in\mathbb{R}^d$ being the fixed parameters and $\vz\in\mathbb{R}^d$ being the true objects. 

The data in our experiment are generated by $y_i=\left|\va_i^{\top}\vz\right|^2+\bm{\epsilon}_i, i\in[m]$, where each coordinate of both the measurement vector $\va_i\in\mathbb{R}^d$ and the true parameter $\vz$ satisfy Gaussian distribution $\mathcal{N}\left(0,0.5\right)$, and $\bm{\epsilon}_i\sim \mathcal{N}\left(0,25\right)$ is the noise.
Here, we set the number of samples $m=1000$ and the dimension $d=10$.

\begin{figure}[!h]
\begin{subfigure}{0.49\linewidth}
  \centering
  \includegraphics[width=0.8\linewidth]{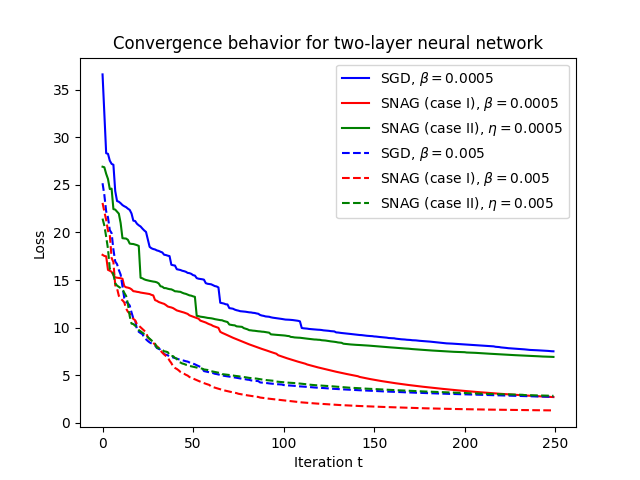}
  \caption{Two-layer neural network}
  \label{fig a}
\end{subfigure}
\begin{subfigure}{0.49\linewidth}
  \centering
  \includegraphics[width=0.8\linewidth]{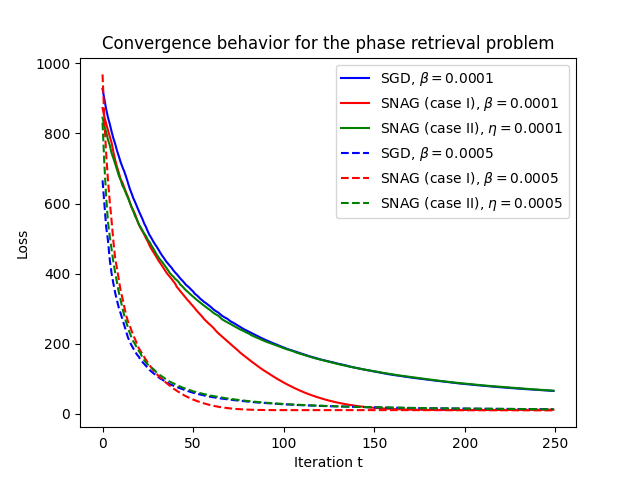}
  \caption{Phase retrieval model}
  \label{fig b}
\end{subfigure}
%\vspace{-1 in}
\caption{Experiment results. We run each algorithm 100 times and plot the average loss at each iteration.}
\label{figure 1}
\end{figure}

\textbf{Experiment Setup}
We set $\beta=\eta$ in Algorithm \ref{algorithm2} (named SNAG (case I) in the legend) and $\lambda_t=\eta$ in Algorithm \ref{algorithm3} (named SNAG (case II) in the legend) where $\eta$ is also the step size of SGD.
The stochastic gradient in each step is computed by samples randomly chosen with batch size 10.
We start the training process with the initial vector satisfying $\mathcal{N}\left(1,25\right)$.

\textbf{Results}
As Figure \ref{figure 1} shows, SGD and SNAG with $\lambda_t=\eta$ exhibit comparable performance under these two possibly non-convex setting, complementing their theoretical analysis.
SNAG (case I) performs best among the three especially with small step size though we only prove its acceleration theoretically in the convex case.

\section{Conclusion} 
In this paper, motivated by several machine learning problems, we propose a new general smoothness, which generalizes the global smoothness and $\left(L_0,L_1\right)$-smoothness.
Under this condition, we analyze NAG method and obtain the accelerated convergence rate of order $\mathcal{O}\left(1/T^2\right)$ for convex optimizations with access to accurate gradients.
For stochastic optimization, we obtain accelerated probabilistic convergence rates of order $\tilde{\mathcal{O}}\left(1/T^2+\sqrt{\left(A+B+C\right)/T}\right)$ under sub-Gaussian relaxed affine-variance noises.
Furthermore, we apply our analysis to smooth optimizations and obtain the result of order $\mathcal{O}\left(\left(B+1\right)/T^2+\sqrt{\left(A+C\right)/T}\right)$ in expectation under expected relaxed affine-variance noises.
All the above derived convergence rates are optimal without further assumptions.

\section*{Declarations}
\subsection*{Funding}
This work was supported in part by the National Natural Science Foundation of China under grant number 12471096, and the National Key Research and Development Program of China under grant number 2021YFA1003500.
\iffalse
\subsection*{Conflict of interest}
The authors have no relevant financial or non-financial interests to disclose.
\fi

\appendix
\section{Examples satisfying the $\left(L_0,L_1,L_2\right)$-smoothness condition}
\label{section B}
\subsection{Two-layer neural networks}  
\label{section two-layer}
Recall the two-layer neural network model in (\ref{two layer}) and we have the following lemma from \cite{taheri2023fast}. 
We refer interested readers to see the proof in their paper.
  \begin{lemma}[Lemma 5 in \cite{taheri2023fast}]
  \label{lemma two layer}
  Let $F$ be in (\ref{two layer}) and $\Phi$ be a two layer neural network with the activation function satisfying that there exist $L,\alpha,l>0$, such that
  \begin{align*}
    \left| \sigma''(t)\right|\leq L, \quad \alpha\leq \sigma'(t)\leq l,\quad \forall t\in\mathbb{R}.
  \end{align*}
  Then, $F$ is self-bounded of gradient and Hessian with constants $h=\frac{lR}{\sqrt{m}}, H=\frac{LR^2}{m^2}+\frac{l^2R^2}{m}$, i.e.,
  \begin{align*}
    \norm{\nabla F(\vx)}\leq h F(\vx),\quad \norm{\nabla^2 F(\vx)}\leq H F(\vx),
  \end{align*}
  where $R=\max\limits_{i\in[n]}\norm{\vw_i}$.
\end{lemma}
In the next lemma, we denote $F^*$ is the minimum of $F(\vx)$ in (\ref{two layer}), i.e., $F(\vx)\geq F^*, \forall \vx\in \mathbb{R}^{\tilde{d}}$.
\begin{lemma}
  Under the condition of Lemma \ref{lemma two layer}, $F\left(\vx\right)$ in (\ref{two layer}) is $\left(L_0,0,L_2\right)$-smooth,
  where $L_0$ and $L_2$ are non-negative constants such that $L_2=\max\left\{h, H\rm{e}\right\}$, $L_0=L_2F^*$.
\end{lemma}
\begin{proof}
    For any $\norm{\vx-\vy}\leq 1/L_2$, define $\gamma(t)=t(\vy-\vx)+\vx, t\in[0,1]$.
    Then, for any $\mu\in[0,1]$ we have,
    \begin{align}
        F(\gamma(\mu))=&\int_{0}^{\mu}\la \nabla F(\gamma(t)), \vy-\vx\ra d t+F(\vx)\notag\\
        \leq &\int_{0}^{\mu} \norm{\nabla F(\gamma(t))}\cdot \norm{\vy-\vx} d t+F(\vx)\notag\\
        \leq &h \norm{\vy-\vx}\int_{0}^{\mu} F(\gamma(t))d t+F(\vx),
    \end{align}
    where the first inequality holds since Cauchy-Schwarz inequality and the second inequality follows from Lemma \ref{lemma two layer}.
    By Gronwall's inequality, we have
    \begin{align}
        F(\gamma(\mu))\leq & F(\vx)\cdot \exp(\mu h\norm{\vy-\vx}), \quad \mu\in[0,1].
    \end{align}
    Moreover, we have 
    \begin{align}
        \nabla F(\vy)-\nabla F(\vx)=&\nabla F(\gamma(1))-\nabla F(\gamma(0))=\int_{0}^{1} \nabla^2 F(\gamma(t))(\vy-\vx) dt,
    \end{align}
    which implies, 
    \begin{align}
        \norm{\nabla F(\vy)-\nabla F(\vx)}=&\norm{\int_{0}^{1} \nabla^2 F(\gamma(t))(\vy-\vx) dt}\notag\\
         \leq & \int_{0}^{1}\norm{\nabla^2 F(\gamma(t))}\norm{\vy-\vx}dt\notag\\
         \leq & H \norm{\vy-\vx}\int_{0}^{1}F(\gamma(t)) dt\notag\\
         \leq & H \norm{\vy-\vx}\int_{0}^{1} F(\vx)\cdot \exp(th\norm{\vy-\vx})d t.
    \end{align}
    Since $\norm{\vy-\vx}\leq \frac{1}{\twolayerL}$, we have 
    \begin{align}
         \norm{\nabla F(\vy)-\nabla F(\vx)}\leq & H F(\vx) \exp(\frac{h}{L_2})\norm{\vy-\vx}\notag\\
         = & \left(H\exp(\frac{h}{L_2})F^*+  H \exp(\frac{h}{L_2})\left(F(\vx)-F^*\right)\right)\norm{\vy-\vx}.
    \end{align}
    By the constraints that $L_2=\max\left\{h, H\text{e}\right\}$, we have 
    \begin{align*}
      L_2\geq H \exp(\frac{h}{L_2}).
    \end{align*}
    Combining with the fact that $F^*$ is positive for the exponential loss, we have
    \begin{align}
        \norm{\nabla F(\vy)-\nabla F(\vx)}\leq \left(L_0+L_2\left(F(\vx)-F^*\right)\right)\norm{\vy-\vx}.
    \end{align}
\end{proof}
\subsection{Phase retrieval model}
We then provide the proof that the phase retrieval model in (\ref{phase retrieval}) satisfying $\left(L_0,L_1,L_2\right)$-smoothness condition.
The following lemma is presented in \cite{chen2023generalized}.
\begin{lemma}
  The function $f(\vx)$ in (\ref{phase retrieval}) belongs to $\mathcal{L}^*_{\text{sym}}\left(\frac{2}{3}\right)$, i.e., for any $\vx, \vy\in\mathbb{R}^d$,
\begin{align}
  \label{phase retrial inequality}
  \norm{\nabla f(\vx)-\nabla f(\vy)}\leq \left(L_0'+L_1'\norm{\nabla f(\vx)}^{\frac{2}{3}}+L_2'\norm{\vx-\vy}^2\right)\norm{\vx-\vy},
\end{align}
where $L_0',L_1',L_2'$ are non-negative constants.
\end{lemma}
Thus, we could derive Lemma \ref{lemma B.4}.
\begin{lemma}
  \label{lemma B.4}
  Suppose that $f(\vx)$ is the phase retrieval model defined in (\ref{phase retrial inequality}). Then, $f(\vx)$ is $(L_0,L_1,0)$-smooth, where $L_0=L_0'+L_2'/L_1^2$ and $L_1=L_1'$.
\end{lemma}
\begin{proof}
  By (\ref{phase retrial inequality}), for any $\vx, \vy\in \mathbb{R}^d$ such that $\norm{\vx-\vy}\leq\frac{1}{L_1}$, we have 
  \begin{align*}
    \norm{\nabla f(\vx)-\nabla f(\vy)}\leq & \left(L_0'+L_1'\norm{\nabla f(\vx)}^{\frac{2}{3}}+L_2'/L_1^2\right)\norm{\vx-\vy}\notag\\
    = & \left(L_0+L_1\norm{\nabla f(\vx)}^{\frac{2}{3}}\right) \norm{\vx-\vy}.
  \end{align*}
\end{proof}

\section{Complementary lemmas}
The following lemma characterizes the relationship between the gradient and the function value gap under the smoothness condition.
Refer to \cite{attia2023sgd} for a proof.
\begin{lemma}
  \label{Lemma A.1}
  Let $f(\cdot): \mathbb{R}^d \rightarrow \mathbb{R}$ be an $L$-smooth function with minimum $f^*$. Then, we have 
  \begin{align*}
    \norm{\nabla f(\vx_t)}^2\leq 2L\left(f\left(\vx_t\right)-f^*\right).
  \end{align*}
\end{lemma}
Lemma \ref{lemma A.2} and Lemma \ref{lemma A.3} are the key to the analysis for $\left(L_0,L_1,L_2\right)$-smooth functions.
\begin{lemma}
    \label{lemma A.2}
    If $f(\cdot): \mathbb{R}^d \rightarrow \mathbb{R}$ is $\left(L_0,L_1,L_2\right)$-smooth with minimum $f^*$, then for any $\vx,\vy\in\mathbb{R}^d$ such that $\norm{\vx-\vy}\leq\min\left\{1/L_1,1/L_2\right\}$, we have
    \begin{align*}
        f(\vy)\leq f(\vx)+\la \nabla f(\vx),\vy-\vx\ra +\frac{L_0+L_1\norm{\nabla f(\vx)}^p+L_2\left(f(\vx)-f(\vx^*)\right)^q}{2}\norm{\vx-\vy}^2.
    \end{align*}
\end{lemma}
\begin{proof}
    \begin{align*}
        &f(\vy)-f(\vx)-\la \vy-\vx, \nabla f\left(\vx\right)\ra\notag\\
        =&\int_{0}^{1}\la \nabla f(\theta \vy+\left(1-\theta\right)\vx)- \nabla f(\vx),\vy-\vx\ra d\theta\notag\\
        \leq & \int_{0}^{1}\norm{\nabla f(\theta \vy+\left(1-\theta\right)\vx)-\nabla f(\vx)}\cdot \norm{\vx-\vy} d\theta\notag\\
        \leq & \int_{0}^{1}\theta \cdot \left(L_0+L_1\norm{\nabla f(\vx)}^p+L_2\left(f(\vx)-f^*\right)^q\right)\norm{\vx-\vy}^2 d \theta\notag\\
        = & \frac{L_0+L_1\norm{\nabla f(\vx)}^p+L_2\left(f(\vx)-f^*\right)^q}{2}\norm{\vx-\vy}^2,
    \end{align*}
    where the first inequality holds since Cauchy-Schwarz inequality and the second inequality follows from the definition of $\left(L_0,L_1,L_2\right)$-smoothness.
\end{proof}
\begin{lemma}
  \label{lemma A.3}
    If $f(\cdot): \mathbb{R}^d \rightarrow \mathbb{R}$ is $\left(L_0,L_1,L_2\right)$-smooth with minimum $f^*$, then we have
    \begin{align}
       \norm{\nabla f(\vx_t)}^2\leq 4L_0\triangle_t+\left(2-p\right)\left(\left(2p\right)^{\frac{p}{2}}2L_1\triangle_t\right)^{\frac{2}{2-p}}+4L_2\triangle_t^{q+1}+8\left(L_1+L_2\right)^2\triangle_t^2,
    \end{align}
    where $\triangle_t=f(\vx_t)-f^*$.
\end{lemma}
\begin{proof}
    Let $\vx=\vx_t-\frac{1}{\Llow}\nabla \fx{t}$. 
    It is easy to verify $\norm{\vx-\vx_t}\leq \min\left\{1/L_1,1/L_2\right\}$.
    By Lemma \ref{lemma A.2}, we have
    \begin{align*}
        f\left(\vx\right)\leq &\fx{t}+\la \nabla \fx{t},\vx-\vx_t\ra+\frac{L_0+L_1\norm{\nabla \fx{t}}^p+L_2\triangle_t^q}{2}\norm{\vx-\vx_t}^2\notag\\
    = &\fx{t}-\frac{\norm{\nabla \fx{t}}^2}{\Llow}\notag\\
    &+\frac{\left(L_0+L_1\norm{\nabla \fx{t}}^p+L_2\triangle_t^q\right)\cdot\norm{\nabla \fx{t}}^2}{2\left(\Llow\right)^2}\notag\\
    \leq & \fx{t}-\frac{\norm{\nabla \fx{t}}^2}{\Llow}\notag\\
    &+\frac{\norm{\nabla \fx{t}}^2}{2\left(\Llow\right)},
    \end{align*}
    which implies
    \begin{align*}
        \frac{\norm{\nabla \fx{t}}^2}{2\left(\Llow\right)}\leq \fx{t}-f\left(\vx\right)\leq \triangle_t.
    \end{align*}
    Re-arranging the above inequality, we obtain that
    \begin{align*}
        \norm{\nabla \fx{t}}^2\leq 2L_0\triangle_t+2L_1\norm{\nabla \fx{t}}^p\triangle_t+2L_2\triangle_t^{q+1}+2\left(L_1+L_2\right)\norm{\nabla \fx{t}}\cdot\triangle_t.
    \end{align*}
    When $p>0$, applying Young's inequality, we have
    \begin{align}
      \label{28}
        \norm{\nabla \fx{t}}^2\leq &2L_0\triangle_t+\frac{1}{4}\norm{\nabla \fx{t}}^2+\left(1-\frac{p}{2}\right)\left(\left(2p\right)^{\frac{p}{2}}2L_1\triangle_t\right)^{\frac{2}{2-p}}\notag\\
        &+2L_2\triangle_t^{q+1}+\frac{1}{4}\norm{\nabla \fx{t}}^2+4\left(L_1+L_2\right)^2\triangle_t^2.
    \end{align}
    Note that (\ref{28}) still holds when $p=0$ since $0^0=1$.
    Hence,
    \begin{align*}
        \norm{\nabla \fx{t}}^2\leq 4L_0\triangle_t+\left(2-p\right)\left(\left(2p\right)^{\frac{p}{2}}2L_1\triangle_t\right)^{\frac{2}{2-p}}+4L_2\triangle_t^{q+1}+8\left(L_1+L_2\right)^2\triangle_t^2.
    \end{align*}
\end{proof}
\begin{corollary}
  \label{corollary 1}
   Let $f(\cdot)$ be an $\left(L_0,L_1,L_2\right)$-smooth function with minimum $f^*$. If $f(\vx_t)-f^*\leq G$, then we have 
   \begin{align*}
    \norm{\nabla f(\vx_t)}^2\leq g(G),
   \end{align*}
   where $g(\cdot): \mathbb{R}\rightarrow\mathbb{R}$ is defined as
   \begin{align}
    \label{g function}
  g(\mu)=4L_0\mu+\left(2-p\right)\left(\left(2p\right)^{\frac{p}{2}}2L_1\mu\right)^{\frac{2}{2-p}}+4L_2\mu^{q+1}+8\left(L_1+L_2\right)^2\mu^2.
\end{align}
\end{corollary}
The following lemma plays crucial role in our probabilistic analysis. Refer to \cite{li2020high} for a proof.
\begin{lemma}[Lemma 1 in \cite{li2020high}]
  \label{lemma high probability}
  Assume that $\{Z_t\}_{t\in[T]}$ is a martingale difference sequence with respect to $\mathbf{\gamma}_1, \mathbf{\gamma}_2, \cdots, \mathbf{\gamma}_T$ and $\mathbb{E}_t[\exp\left(Z_t^2/\sigma_t^2\right)]\leq \exp(1)$ for all $1 \leq t \leq T$, where $\sigma_t$ is a sequence of measurable random variables with respect to $\mathbf{\gamma}_1, \mathbf{\gamma}_2, \cdots, \mathbf{\gamma}_{t-1}$. Then, for any fixed $\lambda >0$ and $\delta\in(0,1)$, with probability at least $1-\delta$, we have
  $$\sumlimits{t=1}{T}Z_t\leq \frac{3\lambda}{4}\sumlimits{t=1}{T}\sigma_t^2+\frac{1}{\lambda}\log\frac{1}{\delta}.$$
\end{lemma}

\section{Non-convex optimization}
\label{section non-convex}
In this section, we re-present SNAG (see Algorithm \ref{algorithm3}) with different step size setting compared to Algorithm \ref{algorithm2} and provide its convergence rate in the non-convex setting.
Algorithm \ref{algorithm3} could reduce to some famous algorithms, such as SGD, and was well studied in \cite{ghadimi2016accelerated,kavis2022high,yu2025convergence}.
Algorithm \ref{algorithm2} can be viewed as a special case of Algorithm \ref{algorithm3}.
%here, particularly the order of $\left\{\lambda_t\right\}_{t\in[T]}$.
\begin{algorithm}[H]
  \caption{Stochastic Nesterov's Accelerated Gradient Descent (SNAG)}
  \label{algorithm3}
  \begin{algorithmic}[1]
  \State {\bfseries Input:} Horizon $T$, $\xag{0}=\vx_0\in \mathbb{R}^d$, step sizes $\left\{\beta_t\right\}_{t\in[T]}, \{\lambda_t\}_{t\in [T]}$, averaging parameters $\left\{\alpha_t\right\}_{t\in[T]}$.
  \For{$t=1,\cdots, T$}
  \State $\xmd{t}=\left(1-\alpha_t\right)\xag{t-1}+\alpha_t \vx_{t-1}$;
  \State \textbf{Set} $\vg_t=\nabla f_\mathbf{z}\left(\vx_t^{md};\rvz_t\right)$;
  \State $\vx_t=\vx_{t-1}-\lambda_t \vg_t$;
  \State $\xag{t}=\xmd{t}-\beta_t \vg_t$.
  \EndFor
  \end{algorithmic}
\end{algorithm}

We have the following results for the above algorithm.
\begin{theorem}	\label{theorem 4}
  Let $T>0$ and $f$ be an $\left(L_0,L_1,L_2\right)$-smooth function.
	Under Assumptions \ref{assumption 1}-\ref{assumption 3},
	consider Algorithm \ref{algorithm3} with $\alpha_t=\frac{2}{t+1},$ $\lambda_t=\eta$ and 
	$\beta_t= \eta \alpha_t + \lambda_t,\forall t\in [T].$
	Let 
	\begin{align*}
	\eta = \min \left\{ {1\over (L_1+L_2) \mathcal{Y}}, \frac{1}{8 \mathcal{W} \left(B \logte+1\right)}, \frac{1}{4\sqrt{A\mathcal{W} T\logte \Ffour}}, \frac{1}{4\sqrt{C\mathcal{W} T\logte}}, \frac{1}{6\mathcal{P}_c^2\logte}\right\},
	\end{align*}
	where 
	\begin{align}
		\label{Y_c}
		\mathcal{Y}= \sqrt{A\logte\Ffour}+\left(\sqrt{B\logte}+1\right)\sqrt{g\left(\Ffour\right)}+\sqrt{C\logte},
	\end{align}
	$$\mathcal{W} = L_0+L_1 \left(g\left(\mathcal{K}\right)\right)^{\frac{p}{2}}+L_2\mathcal{K}^q,$$
	\begin{align}
		\label{P}
		\mathcal{P}_c=\sqrt{A\Ffour+B g\left(\Ffour\right)+C},
	\end{align}
	$$
\mathcal{K}=	\Ffour+\frac{1}{L_1+L_2}\sqrt{g(\Ffour)}+\frac{L_0+L_1\left(g(\Ffour)\right)^{\frac{p}{2}}+L_2\Ffour^q}{2\left(L_1+L_2\right)^2},
	$$
	$$
	\Ffour=  \trianglemd{1}+1+\frac{1}{L_1+L_2}\sqrt{g(1+\trianglemd{1})}+\frac{L_0+L_1\left(g(1+\trianglemd{1})\right)^{\frac{p}{2}}+L_2\left(1+\trianglemd{1}\right)^q}{2\left(L_1+L_2\right)^2},
	$$
	and $g$ is the function given by (\ref{g function}).
	Then with probability at least $1-2\delta,$
	\begin{align}
		\label{nonconvex convergence}
		\frac{1}{T}\sumlimits{l=1}{T} \norm{\nablaf{l}{md}}^2\leq & \frac{2\left(1+\trianglemd{1}\right)}{T}\left(L_1+L_2\right)\mathcal{Y}+ \frac{16\left(1+\trianglemd{1}\right)}{T} \mathcal{W}
		\left(B\logte+1\right)\notag\\
		&+ \frac{8\left(1+\trianglemd{1}\right)}{\sqrt{T}}\left(\sqrt{A \Ffour} + \sqrt{C}\right) \sqrt{\mathcal{W} \logte} + \frac{12\left(1+\trianglemd{1}\right)}{T} \mathcal{P}_c^2\logte.
	\end{align}
	%$$\eta\leq \frac{1}{\left(L_1+L_2\right)\left(\sqrt{A\logte\Ffour}+\left(\sqrt{B\logte}+1\right)\sqrt{g\left(\Ffour\right)}+\sqrt{C\logte}\right)}$$
	%  $$\eta\leq \frac{1}{8 \mathcal{W} \left(B \logte+1\right)}$$
	%  $$\eta\leq \frac{1}{4\sqrt{A\mathcal{W} T\logte \Ffour}}$$
	%  $$\eta\leq \frac{1}{4\sqrt{\mathcal{W}C T \logte}}$$
	%  $$\eta\leq \frac{1}{6\mathcal{P}_c^2\logte}$$
\end{theorem}
The upper rate from (\ref{nonconvex convergence}) is of order $\tilde{O}({1/T} + \sqrt{(A+C)/T})$, which matches that in \cite{ghadimi2016accelerated} for SNAG with bounded variances  and also
the lower rate in \cite{arjevani2023lower} of finding stationary points in non-convex smooth stochastic optimizations with bounded variances when $C>0$.

Under the $\left(L_0,L_1\right)$-smoothness assumption, \cite{yu2025convergence} analyzed SNAG for non-convex objective functions and they proved that the average of the squared norm converges at the rate of $\tilde{\mathcal{O}}\left(1/T+\sqrt{\left(A+C\right)/T}\right)$ with high probability.
Here, we follow the analytical approach from \cite{yu2025convergence} and make slight modifications to the proof methods to accommodate the more general smooth assumptions.
To prove the theorem,
we first provide several useful lemmas following from \cite{ghadimi2016accelerated,kavis2022high,yu2025convergence}.
\begin{proposition}[Proposition 5.2 in \cite{kavis2022high}]
  \label{proposition 1}
  Denote  $\varGamma_t=\left(1-\alpha_t\right)\varGamma_{t-1}$ with $\alpha_t=\frac{2}{t+1}$ and $\varGamma_1=1$, $\forall t\in [T]$. We have that for all $t\in[T]$,
  \begin{align}
    \label{proposition 1.1}
   \varGamma_{t}\sumlimits{k=1}{t}\frac{\alpha_k}{\varGamma_k}=1,
  \end{align}
  and 
  \begin{align}
    \label{proposition 1.2}
    \left[\sumlimits{k=t}{T}\left(1-\alpha_k\right)\varGamma_k\right]\frac{\alpha_t}{\varGamma_t}\leq 2.
  \end{align}
\end{proposition}

\begin{lemma}
  \label{high probability sum non-convex constant}
  Given $T\geq 1 $ and $\delta\in \left(0,1\right)$, if Assumptions \ref{assumption 2} and \ref{assumption 3} hold, then with probability at least $1-\delta$, 
  \begin{align}
    \label{inequality high probability non-convex constant}
    \sumlimits{k=1}{l}-\la \nablaf{k}{md},\rvxi{k} \ra\leq \frac{1}{4}\sumlimits{k=1}{l}\frac{\mathcal{P}_k^2}{\mathcal{P}_c^2}\norm{\nablaf{k}{md}}^2+3\mathcal{P}_c^2\log\frac{T}{\delta}, \quad \forall l\in[T],
  \end{align}
  where 
  \begin{align}
    \label{P_t}
    \mathcal{P}_k=\sqrt{A\trianglemd{k}+B g\left(\trianglemd{k}\right)+C}, 
  \end{align}
and $ \mathcal{P}_c$ is given by (\ref{P}).
\end{lemma}

\begin{proof}
  Let $Z_k=-\la \nablaf{k}{md},\rvxi{k}\ra$. 
Note that $\nablaf{k}{md}$ is a random variable dependent on $\rvz_1,\cdots,\rvz_{k-1}$ and $\rvxi{k}$ is dependent on $\rvz_1,\cdots,\rvz_k$. Therefore, it is apparent that $Z_k$ is a martingale difference sequence since
$$\mathbb{E}\left[- \la \nablaf{k}{md},\rvxi{t}\ra|\rvz_1,\cdots,\rvz_{k-1}\right]=-\la \nablaf{k}{md},\mathbb{E}_k[\rvxi{k}]\ra=0.$$
Also by Assumption \ref{assumption 3} and applying Cauchy-Schwarz inequality, we obtain that
\begin{align*}
	&\mathbb{E}_k\left[\exp\left(\frac{Z_k^2}{\norm{\nablaf{k}{md}}^2\left(A\trianglemd{k}+B\norm{\nablaf{k}{md}}^2+C\right)}\right)\right]\notag\\
  \leq & \mathbb{E}_k \left[\exp\left(\frac{\norm{\nablaf{k}{md}}^2\norm{\rvxi{k}}^2}{\norm{\nablaf{k}{md}}^2\left(A\trianglemd{k}+B\norm{\nablaf{k}{md}}^2+C\right)}\right)   \right] \leq \text{e}.
\end{align*}
Therefore, given any $l\in [T]$, applying Lemma \ref{lemma high probability}, we have that for any $\lambda>0$, with probability at least $1-\delta$,
\begin{align*}
	\sumlimits{k=1}{l}Z_t\leq &\frac{3\lambda}{4}\sumlimits{k=1}{l}\norm{\nablaf{k}{md}}^2\left(A\trianglemd{k}+B\norm{\nablaf{k}{md}}^2+C\right)+\frac{1}{\lambda}\log\frac{1}{\delta}\notag\\
  \leq & \frac{3\lambda}{4}\sumlimits{k=1}{l}\norm{\nablaf{k}{md}}^2\mathcal{P}_k^2+\frac{1}{\lambda}\log\frac{1}{\delta}
\end{align*}
where $\mathcal{P}_k$ is defined in (\ref{P_t}).
For any fixed $\lambda$, we can  re-scale over $\delta$ and have that  with probability at least $1-\delta$, for all $l\in[T]$,
\begin{align}
	\label{4.14}
	\sumlimits{k=1}{l}-\la\nablaf{k}{md},\rvxi{k}\ra\leq \frac{3\lambda}{4}\sumlimits{t=1}{l}\norm{\nablaf{k}{md}}^2 \mathcal{P}_k^2+\frac{1}{\lambda}\log\frac{T}{\delta}.
\end{align}
Let $\lambda=\frac{1}{3 \mathcal{P}_c^2}$, and we obtain the desired result.
\end{proof}

\begin{proposition}
 	\label{proposition overline x_t - x_t}
 	Let $\{\vx_t\}_{t\in[T]}$ and $\{\xmd{t}\}_{t\in[T]}$ be generated by Algorithm \ref{algorithm3}. We have
 	\begin{align}
 		\label{overline x_t-x_t}
 		\xmd{t}-\vx_{t-1} =\left(1-\alpha_t\right)\varGamma_{t-1}\sumlimits{k=1}{t-1}\frac{\alpha_k}{\varGamma_k}\frac{\left(\lambda_k-\beta_k\right)}{\alpha_k}\vg_k,
 	\end{align}
 	and
 	\begin{align}
    \label{eq 91}
 		\norm{\xmd{t}-\vx_{t-1}}^2\leq \left(1-\alpha_t\right)\varGamma_t\sumlimits{k=1}{t-1}\frac{\alpha_k}{\varGamma_k}\frac{\left(\lambda_k-\beta_k\right)^2}{\alpha_k^2}\norm{\vg_k}^2.
 	\end{align}
 \end{proposition}
\begin{proof}
  From Algorithm \ref{algorithm3}, we have
  \begin{align*}
    \xag{k}-\vx_{k}=  \xmd{k}-\beta_k \vg_k-\vx_{k-1}+\lambda_k \vg_k
    =  \left(1-\alpha_k\right)\left(\xag{k-1}-\vx_{k-1}\right)+\left(\lambda_k-\beta_k\right)\vg_k.
  \end{align*}
Since $\xag{0}=\vx_0$, we obtain that
\begin{align*}
  \xag{k}-\vx_{k} = & \sumlimits{i=1}{k} \left(\prod_{j=i+1}^k\left(1-\alpha_j\right) \right)  \left(\lambda_i-\beta_i\right)\vg_i
   =  \varGamma_k\sumlimits{i=1}{k} \frac{1}{\varGamma_i}\left(\lambda_i-\beta_i\right)\vg_i.
\end{align*}
Taking the norm function on both sides and applying the triangle inequality, we have
\begin{align}
  \label{eq 92}
  \norm{\xag{k}-\vx_{k}} \leq & \varGamma_k\sumlimits{i=1}{k}\frac{1}{\varGamma_i}\left|\lambda_i-\beta_i\right| \cdot \norm{\vg_i}
  =  \varGamma_k \sumlimits{i=1}{k}\frac{\alpha_i}{\varGamma_i}\frac{\left|\lambda_i-\beta_i\right|}{\alpha_i} \cdot \norm{\vg_i}.
\end{align}
By the iteration step in Algorithm \ref{algorithm3}, we have 
\begin{align*}
  \xmd{k}-\vx_{k-1}=\left(1-\alpha_k\right)\left(\xag{k-1}-\vx_{k-1}\right).
\end{align*}
Combining with (\ref{eq 92}), we obtain that
\begin{align*}
  \norm{\xmd{k}-\vx_{k-1}}= & \left(1-\alpha_k\right) \norm{\xag{k-1}-\vx_{k-1}}
  \leq  \left(1-\alpha_k\right)\varGamma_{k-1}\sumlimits{i=1}{k-1}\frac{\alpha_i}{\varGamma_i}\frac{\left|\lambda_i-\beta_i\right|}{\alpha_i} \cdot \norm{\vg_i}.
\end{align*}
Similarly, by the convexity of norm square and (\ref{proposition 1.1}), 
	\begin{align*}
 		\norm{\xmd{k}-\vx_{k-1}}^2\leq & \left(1-\alpha_k\right)^2\varGamma_{k-1}\sumlimits{i=1}{k-1}\frac{\alpha_i}{\varGamma_i}\frac{\left(\lambda_i-\beta_i\right)^2}{\alpha_k^2}\norm{\vg_i}^2\notag\\
    = & \left(1-\alpha_k\right)\varGamma_k\sumlimits{i=1}{k-1}\frac{\alpha_i}{\varGamma_i}\frac{\left(\lambda_i-\beta_i\right)^2}{\alpha_k^2}\norm{\vg_i}^2.
 	\end{align*}
 \end{proof}

 \begin{lemma}
	\label{sqrt sum}
	Let $\{a_t\}_{t\in[n]}$ be a sequence of non-negative real numbers. We have
	\begin{align*}
		\sqrt{\sum_{i=1}^n a_i}\leq \sumlimits{i=1}{n}\sqrt{a_i}.
	\end{align*}
\end{lemma}

In the following analysis, denote $\triangle_t=f\left(\vx_t\right)-f^*$ for simplicity.

\begin{lemma}
  \label{nonconvex bound}
  Under the conditions and notations of Theorem \ref{theorem 4}, $\trianglemd{t}\leq \Ffour, \forall t\in[T]$, hold with probability at least $1-\delta$.
  %, where  $\mathcal{P}_t$, $\mathcal{P}_c$ are in \ref{P_t}, \ref{P} and $\Ffour$  is defined as
%  \begin{align}
%    \label{Gconstant}
%    =.
%  \end{align} 
\end{lemma}

\begin{proof}
   We assume that (\ref{eq 32}) and (\ref{inequality high probability non-convex constant}) always happen and then deduce $\trianglemd{t}\leq \Ffour$ for all $t\in[T]$. 
  Since (\ref{eq 32}) and (\ref{inequality high probability non-convex constant}) happen with probability at least $1-\delta$ separately, $\trianglemd{t}\leq \Ffour$, $\forall t\in[T]$, holds with probability at least $1-2\delta$. 
  It is obvious that $f(\xmd{1})-f^*\leq \Ffour$. Therefore, by Corollary \ref{corollary 1} we have $\mathcal{P}_1\leq \mathcal{P}_c$.
  Suppose that for some $t\in[T]$, $$f(\xmd{l})-f^*\leq \Ffour,\quad\forall l\in[t].$$ 
  By the triangle inequality of the norm function, we have that for all $l\in[t]$,
  \begin{align}
    \label{eq 94}
    \norm{\vg_l}\leq & \norm{\vg_l-\nablaf{l}{md}}+\norm{\nablaf{l}{md}}\notag\\
    \leq & \sqrt{\left(A\trianglemd{l}+B\norm{\nablaf{l}{md}}^2+C\right)\logte}+\norm{\nablaf{l}{md}}\notag\\
    \leq & \sqrt{A\logte\trianglemd{l}}+\left(\sqrt{B\logte}+1\right)\norm{\nablaf{l}{md}}+\sqrt{C\logte},
   \end{align}
  where the second inequality follows from Lemma \ref{lemma C.1} and the last inequality follows from Lemma \ref{sqrt sum}.
  Combining with Corollary \ref{corollary 1} and the assumption that $\trianglemd{l}\leq \Ffour, \forall l\in[t]$, we have
  \begin{align} 
    \label{g_t}
    \norm{\vg_l} \leq  \mathcal{Y}, \quad \forall l\in[t].
  \end{align}
    By the iteration step of Algorithm \ref{algorithm3}, we have
  \begin{align*}
  \norm{\vx_l-\vx_{l-1}}=\lambda_l\norm{\vg_l}=\eta \norm{\vg_l}\leq \eta \mathcal{Y}\leq \min\left\{1/L_1,1/L_2\right\}, \quad \forall l\in[t],
  \end{align*}
  where the last inequality follows from the restriction of $\eta$.
 Thus, we could apply Lemma \ref{lemma A.2} and obtain that
\begin{align*}
   f(\vx_l)-f(\vx_{l-1})\leq & \la \nabla f(\vx_{l-1}),\vx_l-\vx_{l-1}\ra+\frac{L_0+L_1\norm{\nabla f(\vx_{l-1})}^p+L_2\triangle_{l-1}^q}{2}\norm{\vx_l-\vx_{l-1}}^2\notag\\
   = & -\eta\la \nabla f(\vx_{l-1}),\vg_l\ra+\frac{L_0+L_1\norm{\nabla f(\vx_{l-1})}^p+L_2\triangle_{l-1}^q}{2}\eta^2\norm{\vg_l}^2\notag\\
   = & -\eta\la \nablaf{l}{md}+\nabla f(\vx_{l-1})-\nablaf{l}{md},\nablaf{l}{md}+\rvxi{l}\ra \notag\\
   & + \frac{L_0+L_1\norm{\nabla f(\vx_{l-1})}^p+L_2\triangle_{l-1}^q}{2}\eta^2\norm{\vg_l}^2\notag\\
   = & -\eta \norm{\nablaf{l}{md}}^2-\eta \la \nablaf{l}{md},\rvxi{l}\ra -\eta \la \nabla f(\vx_{l-1})-\nablaf{l}{md},\vg_l\ra \notag\\
   &+ \frac{L_0+L_1\norm{\nabla f(\vx_{l-1})}^p+L_2\triangle_{l-1}^q}{2}\eta^2\norm{\vg_l}^2\notag\\
  \leq & -\eta \norm{\nablaf{l}{md}}^2-\eta\la \nablaf{l}{md},\rvxi{l}\ra +\eta \norm{\nabla f(\vx_{l-1})-\nablaf{l}{md}}\norm{\vg_l} \notag\\
  & +\frac{L_0+L_1\norm{\nabla f(\vx_{l-1})}^p+L_2\triangle_{l-1}^q}{2}\eta^2\norm{\vg_l}^2,
\end{align*} 
where the first equation follows from the update rule in Algorithm \ref{algorithm3} and the last line follows from Cauchy-Schwarz inequality. 
Applying (\ref{proposition overline x_t - x_t}) with  $\frac{\beta_t-\lambda_t}{\alpha_t}=\lambda_t=\eta$, we have 
\begin{align}
  \label{bound overline x_t-x_t}
  \norm{\xmd{l}-\vx_{l-1}} = & \left(1-\alpha_l\right)\varGamma_{l-1} \norm{\sumlimits{k=1}{l-1}\frac{\alpha_k}{\varGamma_k}\frac{\left(\lambda_k-\beta_k\right)}{\alpha_k}\vg_k}
  \leq  \left(1-\alpha_l\right)\varGamma_{l-1} \sumlimits{k=1}{l-1} \frac{\alpha_k}{\varGamma_k} \eta \norm{\vg_k}\notag\\
  \leq &  \eta \mathcal{Y}\varGamma_{t-1}\sumlimits{k=1}{l-1}\frac{\alpha_k}{\varGamma_k}\leq \min\left\{1/L_1,1/L_2\right\},
\end{align}
where the first inequality follows from the triangle inequality and the second inequality holds since (\ref{eq 94}).
The last inequality follows from (\ref{proposition 1.1}).
Note that $\norm{\vg_l}\leq \mathcal{Y}$ for all $l\in[t]$ and $\norm{\xmd{l}-\vx_{l-1}}$ depends on $\vg_1, \cdots, \vg_{l-1}$.
Thus, (\ref{bound overline x_t-x_t}) holds for all $l\in[t+1]$. 
Applying Definition \ref{definition 1}, we have that 
\begin{align}
  \label{eq 96}
    & f(\vx_l)-f(\vx_{l-1}) \notag\\
  \leq & -\eta \norm{\nablaf{l}{md}}^2-\eta\la \nablaf{l}{md},\rvxi{l}\ra +\frac{L_0+L_1\norm{\nabla f(\vx_{l-1})}^p+L_2\triangle_{l-1}^q}{2}\eta^2\norm{\vg_l}^2 \notag\\
   & +\eta \left(L_0+L_1\norm{\nabla f(\vx_{l-1})}^p+L_2\triangle_{l-1}^q \right)\norm{\vx_{l-1}-\xmd{l}}\norm{\vg_l}\notag\\
  \leq & -\eta \norm{\nablaf{l}{md}}^2-\eta\la \nablaf{l}{md},\rvxi{l}\ra + \left(L_0+L_1\norm{\nabla f(\vx_{l-1})}^p+L_2\triangle_{l-1}^q\right) \eta^2\norm{\vg_l}^2\notag\\
  &+ \frac{L_0+L_1\norm{\nabla f(\vx_{l-1})}^p+L_2\triangle_{l-1}^q}{2}\norm{\vx_{l-1}-\xmd{l}}^2 \notag\\
  \leq & -\eta \norm{\nablaf{l}{md}}^2-\eta\la \nablaf{l}{md},\rvxi{l}\ra + \eta^2\left(L_0+L_1\norm{\nabla f(\vx_{l-1})}^p+L_2\triangle_{l-1}^q\right) \norm{\vg_l}^2\notag\\
  &+ \frac{\eta^2}{2}\left(L_0+L_1\norm{\nabla f(\vx_{l-1})}^p+L_2\triangle_{l-1}^q\right) \left(1-\alpha_l\right)\varGamma_l\sumlimits{k=1}{l-1}\frac{\alpha_k}{\varGamma_k}\norm{\vg_k}^2,
\end{align}
where the second inequality follows from the fact that $ab\leq \frac{a^2+b^2}{2}$ and the last inequality follows from (\ref{eq 91}).
Summing up the above inequality over $l\in[t]$, we obtain that
\begin{align}
  \label{eq 100}
  &f(\vx_t)-f(\vx_0) \notag\\
  \leq & \frac{\eta^2}{2} \sumlimits{l=1}{t}\left[\left(L_0+L_1\norm{\nabla f(\vx_{l-1})}^p+L_2\triangle_{l-1}^q\right) \left(1-\alpha_l\right)\varGamma_l\sumlimits{k=1}{l}\frac{\alpha_k}{\varGamma_k}\norm{\vg_k}^2\right]\notag\\
  &+ \eta^2\sumlimits{l=1}{t}\left(L_0+L_1\norm{\nabla f(\vx_{l-1})}^p+L_2\triangle_{l-1}^q\right) \norm{\vg_l}^2-\eta\sumlimits{l=1}{t}\norm{\nablaf{l}{md}}^2-\eta\sumlimits{l=1}{t}\la \nablaf{l}{md},\rvxi{l}\ra\notag\\
  \leq  & \frac{\eta^2}{2} \sumlimits{l=1}{t} \left[\sumlimits{k=l}{t}\left(L_0+L_1\norm{\nabla f(\vx_{k-1})}^p+L_2\triangle_{k-1}^q\right) \left(1-\alpha_k\right)\varGamma_k \right] \frac{\alpha_l}{\varGamma_l}\norm{\vg_l}^2\notag\\
 &+ \eta^2\sumlimits{l=1}{t}\left(L_0+L_1\norm{\nabla f(\vx_{l-1})}^p+L_2\triangle_{l-1}^q\right) \norm{\vg_l}^2-\eta\sumlimits{l=1}{t}\norm{\nablaf{l}{md}}^2-\eta\sumlimits{l=1}{t}\la \nablaf{l}{md},\rvxi{l}\ra.
\end{align}
\iffalse
Applying the triangle inequality of $\norm{\cdot}$, we have that for all $l\in[t+1]$,
\begin{align*}
  \norm{\nabla f(\vx_{l-1})}\leq & \norm{\nablaf{l}{md}-\nabla f(\vx_{l-1})}+\norm{\nablaf{l}{md}}\notag\\
  \leq & \left(L_0+L_1\norm{\nablaf{l}{md}}^p+L_2\left(\trianglemd{l}\right)^q\right)\norm{\xmd{l}-\vx_{l-1}}+\norm{\nablaf{l}{md}},
\end{align*}
where the last inequality follows from the Definition \ref{definition 1} with $\norm{\xmd{l}-\vx_{l-1}}\leq\min\left\{1/L_1,1/L_2\right\}, \forall l\in[t+1],$ proved in (\ref{bound overline x_t-x_t}).
Combining with the assumption that $\trianglemd{l}\leq \Ffour, \forall l\in[t]$ and Corollary \ref{corollary 1}, we have 
\begin{align*}
  \norm{\nabla f(\vx_{t-1})}\leq & \frac{1}{L_1+L_2}\left(L_0+L_1 \left(g(\Ffour)\right)^\frac{p}{2}+L_2 \Ffour^q\right)+\sqrt{g(\Ffour)}=\mathcal{K}.
\end{align*}
\fi
By (\ref{bound overline x_t-x_t}), we have that $\norm{\xmd{l}-\vx_{l-1}}\leq \min\left\{1/L_1,1/L_2\right\}$ for all $l\in[t+1]$.
Thus, applying Lemma \ref{lemma A.2} again, we obtain that
\begin{align*}
  &f(\vx_{l-1})\notag\\
  \leq & f(\xmd{l})+\la \nablaf{l}{md}, \vx_{l-1}-\xmd{l} \ra + \frac{L_0+L_1\norm{\nablaf{l}{md}}^p+L_2\left(\trianglemd{l}\right)^q}{2}\norm{\vx_{l-1}-\xmd{l}}^2\notag\\
  \leq & f(\xmd{l})+\norm{\nablaf{l}{md}}\cdot \norm{ \vx_{l-1}-\xmd{l}} +\frac{L_0+L_1\norm{\nablaf{l}{md}}^p+L_2\left(\trianglemd{l}\right)^q}{2}\norm{\vx_{l-1}-\xmd{l}}^2,
\end{align*}
where the second inequality follows from Cauchy-Schwarz inequality.
Subtracting $f^*$ from both sides and applying the assumption that $\trianglemd{l}\leq \Ffour, \forall l\in[t]$, we have 
\begin{align*}
  f(\vx_{l-1})-f^*\leq \Ffour+\frac{1}{L_1+L_2}\sqrt{g(\Ffour)}+\frac{L_0+L_1\left(g(\Ffour)\right)^{\frac{p}{2}}+L_2\Ffour^q}{2\left(L_1+L_2\right)^2}=\mathcal{K}, \quad \forall l\in[t].
\end{align*}
Thus, by Corollary \ref{corollary 1}, we have $\norm{\nabla f(\vx_{l})}\leq \sqrt{g(\mathcal{K})}$ for all $l\in[t-1]$.
Combining with (\ref{eq 100}), we obtain that 
\begin{align}
  \label{eq 101}
  f(\vx_t)-f(\vx_0)\leq &\frac{\eta^2}{2} \mathcal{W} \sumlimits{l=1}{t} \left[\sumlimits{k=l}{t}\left(1-\alpha_k\right)\varGamma_k\right] \frac{\alpha_l}{\varGamma_l}    \norm{\vg_l}^2  + \eta^2 \mathcal{W} \sumlimits{l=1}{t}\norm{\vg_l}^2 \notag\\
  &  -\eta\sumlimits{t=1}{l}\norm{\nablaf{k}{md}}^2-\eta\sumlimits{t=1}{l}\la \nablaf{k}{md},\rvxi{t}\ra\notag\\
  \leq & 2\eta^2  \mathcal{W} \sumlimits{l=1}{t}\norm{\vg_l}^2 -\eta\sumlimits{l=1}{t}\norm{\nablaf{l}{md}}^2-\eta\sumlimits{l=1}{t}\la \nablaf{l}{md},\rvxi{l}\ra,
\end{align}
where the second inequality follows from (\ref{proposition 1.2}).
Using the fact that $\norm{\va+\vb}^2\leq 2\norm{\va}^2+2\norm{\vb}^2$ and applying (\ref{eq 32}), we have that for all $l\in[t]$,
\begin{align*}
  \norm{\vg_l}^2\leq & 2\norm{\rvxi{l}}^2+2\norm{\nablaf{k}{md}}^2\notag\\
  \leq& 2\left(A\trianglemd{l}+B\norm{\nablaf{l}{md}}^2+C\right)\logte+2\norm{\nablaf{l}{md}}^2\notag\\
  =& 2\left(A\logte \trianglemd{l}+\left(B \logte+1\right)\norm{\nablaf{l}{md}}^2+C\logte\right).
\end{align*}
Combining with (\ref{eq 101}) and applying Lemma \ref{high probability sum non-convex constant} to the summation of the martingale difference sequence, we obtain that
\begin{align}
  \label{eq 101.1}
  f(\vx_t)-f(\vx_0) \leq & 4\eta^2 \mathcal{W}\left(B\logte+1\right) \sumlimits{l=1}{t} \norm{\nablaf{l}{md}}^2+ 4\eta^2  \mathcal{W} A\logte \sumlimits{l=1}{t} \trianglemd{l}\notag\\
  &+4\eta^2 t  \mathcal{W} C\logte - \eta\sumlimits{l=1}{t}\norm{\nablaf{l}{md}}^2 + \frac{1}{4}\eta\sumlimits{l=1}{t}\frac{\mathcal{P}_l^2}{\mathcal{P}_c^2}\norm{\nablaf{l}{md}}^2+3\eta \mathcal{P}_c^2\log\frac{T}{\delta}\notag\\
  \leq & -{\eta \over 2} \sumlimits{l=1}{t} \norm{\nablaf{l}{md}}^2+ \frac{1}{4}+\frac{1}{4}+\frac{1}{2}.
\end{align}
Since $\xmd{1}=\left(1-\alpha_1\right)\xag{0}+\alpha_1 \vx_0$ and $\xag{0}=\vx_0$, we have $f(\vx_0)=f(\xmd{1})$.
Thus, 
\begin{align}
  \label{eq 102}
  \triangle_t\leq \trianglemd{1}+1.
\end{align}
Since (\ref{bound overline x_t-x_t}) holds for all $l\in [t+1]$, we have that 
\begin{align*}
  \norm{\xmd{t+1}-\vx_t}\leq \min\left\{1/L_1,1/L_2\right\}.
\end{align*}
Therefore, applying Lemma \ref{lemma A.2} again, we obtain that 
\begin{align*}
  f(\xmd{t+1})\leq & f(\vx_t)+\la \nabla f(\vx_{t}), \xmd{t+1}-\vx_t\ra+\frac{L_0+L_1\norm{\nabla f(\vx_t)}^p+L_2\triangle_{t}^q}{2}\norm{\xmd{t+1}-\vx_t}^2\notag\\
  \leq & f(\vx_t)+ \norm{\nabla f(\vx_{t})}\cdot \norm{\xmd{t+1}-\vx_t}+\frac{L_0+L_1\norm{\nabla f(\vx_t)}^p+L_2\triangle_{t}^q}{2}\norm{\xmd{t+1}-\vx_t}^2,
\end{align*}
where the second inequality follows from Cauchy-Schwarz inequality.
Subtracting $f^*$ from both sides and combining with (\ref{eq 102}), we have 
\begin{align}
  \trianglemd{t+1}\leq & \trianglemd{1}+1+\frac{1}{L_1+L_2}\sqrt{g(1+\trianglemd{1})}+\frac{L_0+L_1\left(g(1+\trianglemd{1})\right)^{\frac{p}{2}}+L_2\left(1+\trianglemd{1}\right)^q}{2\left(L_1+L_2\right)^2} \leq \Ffour.
\end{align}
Now we finish the induction and obtain the desired result.

\end{proof}

\begin{proof}[Proof of Theorem \ref{theorem 4}]
  From Lemma \ref{nonconvex bound}, we have that with probability at least $1-2\delta$, $\trianglemd{t}\leq \Ffour$ for all $t\in[T]$.
  Thus, (\ref{eq 101.1}) holds when $t=T$, i.e.,
  \begin{align}
   {\eta \over 2} \sumlimits{l=1}{T} \norm{\nablaf{l}{md}}^2\leq 1+ \trianglemd{1}.  
  \end{align}
  Dividing $T \eta/2$ on both sides and combining with the constraints of $\eta$, we get the desired results.
\end{proof}

\section{Omitted proof}
\label{section omitted proof}
\begin{proof}[Proof of Lemma \ref{lemma 2.1}]
  To start with, we will prove the first line by induction.
  It is obvious that the inequality holds for $B_0=0$. 
  Suppose that for some $0\leq t\leq T$, we have
  \begin{align*}
    \frac{1}{4}k^2\leq B_k\leq k^2, \forall k\in[t].
  \end{align*}
  Then, we have 
  \begin{align*}
    B_{t+1}\leq t^2+\frac{1}{2}\left(1+\sqrt{4t^2+1}\right)\leq t^2+\frac{1}{2}\left(1+2t+1\right)\leq (t+1)^2,
  \end{align*}
  and 
  \begin{align*}
    B_{t+1}\geq \frac{1}{4}t^2+\frac{1}{2}\left(1+\sqrt{t^2+1}\right)\geq \frac{1}{4}\left(t+1\right)^2.
  \end{align*}
  Therefore, we finish the proof for  $\frac{1}{4}t^2\leq B_t\leq t^2, \forall t\in[T]$.
For the second conclusion in Lemma \ref{lemma 2.1},
  \begin{align*}
    \left(A_t-A_{t-1}\right)^2=\left(B_t-B_{t-1}\right)^2=&\frac{1}{4}\left(1+2\sqrt{4B_{t-1}+1}+4B_{t-1}+1\right)\notag\\
    = & B_{t-1}+\frac{1}{2}\left(1+\sqrt{4B_{t-1}+1}\right)\notag\\
    = & B_t.
  \end{align*}
Since $B_t\geq \frac{1}{4}t^2, \forall t\in[T]$, we have $B_t\geq 0, \forall t\in[T]$. Therefore, 
  \begin{align*}
    A_t-A_{t-1}= B_t-B_{t-1}=\frac{1}{2}+\frac{1}{2}\sqrt{4B_{t-1}+1}\geq 1.
  \end{align*}
Now we finish the proof for all the inequalities.
\end{proof}

\bibliographystyle{abbrv}
\bibliography{references}

\end{document}